\documentclass[12pt,reqno]{amsart}

\usepackage{amssymb}
\usepackage{amsmath}
\usepackage{amsthm}
\usepackage{bbm}
\usepackage{graphics}
\usepackage{graphicx}
\usepackage{float}
\usepackage{verbatim}
\usepackage{hyperref}
\usepackage{fullpage}

\newtheorem{thm}{Theorem}[section]
\newtheorem*{Warnke}{Warnke's inequality}
\newtheorem*{Janson}{Janson's inequality}

\newtheorem*{FKG}{The FKG inequality}
\newtheorem{lemma}[thm]{Lemma}

\newtheorem{prop}[thm]{Proposition}
\newtheorem{obs}[thm]{Observation}

\newtheorem{fact}[thm]{Fact}

\theoremstyle{remark}

\theoremstyle{definition}

\newtheorem{defn}[thm]{Definition}

\newcommand{\ds}{\displaystyle}

\newcommand{\Var}{\operatorname{Var}}

\def\SF{\textup{SF}}

\def\B{\mathcal{B}}
\def\C{\mathcal{C}}

\def\E{\mathcal{E}}

\def\G{\mathcal{G}}
\def\HH{\mathcal{H}}

\def\M{\mathcal{M}}

\def\O{\mathcal{O}}
\def\P{\mathcal{P}}

\def\S{\mathcal{S}}

\def\Ex{\mathbb{E}}
\def\Exp{\mathbb{E}}
\def\N{\mathbb{N}}
\def\Pr{\mathbb{P}}

\def\Z{\mathbb{Z}}

\def\eps{\varepsilon}
\def\le{\leqslant}
\def\ge{\geqslant}

\title[The sharp threshold for sum-free sets in even-order abelian groups]{The sharp threshold for maximum-size sum-free subsets in even-order abelian groups}
\author{Neal Bushaw \and
        Maur\'icio Collares Neto \and
        Robert Morris \and
        Paul Smith}
        \address{
    School of Mathematics and Statistics,
Arizona State University,
Tempe, AZ 85287 USA
 } \email{neal@asu.edu}
 \address{
    IMPA, Estrada Dona Castorina 110, Jardim Bot\^anico, Rio de Janeiro, RJ, Brasil
 } \email{collares|rob|psmith@impa.br}
 \thanks{Research supported in part by a CAPES bolsa Proex (MCN), a CNPq bolsa PDJ (PS) and a CNPq bolsa de Produtividade em Pesquisa (RM)}

\begin{document}

\begin{abstract}
We study sum-free sets in sparse random subsets of even-order abelian groups. In particular, we determine the sharp threshold for the following property: the largest such set is contained in some maximum-size sum-free subset of the group. This theorem extends recent work of Balogh, Morris and Samotij, who resolved the case $G = \Z_{2n}$, and who obtained a weaker threshold (up to a constant factor) in general.
\end{abstract}

\maketitle

\section{Introduction}

In recent years, great advances have been made in the study of the extremal and structural properties of sparse random sets. For example, the threshold functions for many classical theorems, such as Szemer\'edi's theorem on arithmetic progressions and Tur\'an's theorem in extremal graph theory, have been determined (see~\cite{BMS2,CG,FRS,ST,Sch}), and in a few cases sharp thresholds have been shown to exist (see, e.g.,~\cite{BMSW,FRRT}). In this paper we will determine the sharp threshold for the maximum sum-free subset problem in an arbitrary even-order abelian group. Our main theorem improves some recent results of Balogh, Morris and Samotij~\cite{BMS1}, who resolved the case~$G = \Z_{2n}$, and obtained weaker bounds in the general setting.

Given an abelian group $G$, we say that a subset $A \subset G$ is \emph{sum-free} if $A \cap (A+A) = \emptyset$, or, equivalently, if there is no solution to the equation $x + y = z$ with $x,y,z \in A$. The study of such sets was introduced by Schur~\cite{Schur} in 1916, and their extremal and structural properties have been extensively studied over the past several decades (see, e.g.,~\cite{AK}). For example, it is easy to see that if $|G| = 2n$ then the largest sum-free subset of $G$ has size $n$ (consider the odd coset of a subgroup of index 2), and in 1969 Diananda and Yap~\cite{DY} extended this simple fact by solving the extremal problem whenever $|G|$ has a prime divisor $q$ with $q \not\equiv 1 \pmod 3$. Nevertheless, more than 30 years passed before the classification was completed by Green and Ruzsa~\cite{GR}. The structure of a typical sum-free subset of an even-order abelian group was determined by Lev, \L uczak and Schoen~\cite{LLS} and Sapozhenko~\cite{Sap02}, and similar results in the set $\{1,\ldots,n\}$ were obtained by Green~\cite{Green} and Sapozhenko~\cite{Sap03}. We refer the reader to~\cite{ABMS1,ABMS2} for some more recent sparse refinements of these results.

The study of sparse random analogues of classical extremal and Ramsey-type results was introduced for graphs by Frankl and R\"odl~\cite{FR86} and Babai, Simonovits and Spencer~\cite{BSS}, and for additive structures by Kohayakawa, {\L}uczak and R\"odl~\cite{KLR1}, and notable early progress was made by R\"odl and Ruci\'nski~\cite{RR1,RR2}. The first result of this type for sum-free sets was obtained by Graham, R\"odl and Ruci\'nski~\cite{GRR}, who determined the threshold function for Schur's theorem. More precisely, they showed that if $p \gg 1/\sqrt{n}$, then almost every $p$-random\footnote{The $p$-random subset of a set $X$, often denoted $X_p$, is obtained by including each element with probability~$p$, independently of all other elements.} subset $A \subset \Z_n$ has the following property: every 2-colouring of $A$ contains a monochromatic Schur triple, i.e., a triple with $x + y = z$. On the other hand, if $p \ll 1/\sqrt{n}$ then with high probability there exist 2-colourings of $A$ that avoid such triples.

In this paper we consider the extremal version of this question, that is, how large is a maximum-size sum-free set in a $p$-random subset of an abelian group? For the group $\Z_{2n}$, this problem was resolved (asymptotically) by Conlon and Gowers~\cite{CG} and Schacht~\cite{Sch}, who determined the following threshold:
\begin{equation}\label{eq:CGS}
\max\big\{ |B| \,:\, B \subset A = (\Z_{2n})_p \textup{ is sum-free} \big\} \,=\, \left\{
\begin{array} {c@{\quad \textup{if} \quad}l}
\big( 1+o(1) \big) \cdot 2pn & p \ll 1/\sqrt{n} \\[+1ex]
\big( 1/2 + o(1)  \big) \cdot 2pn & p \gg 1/\sqrt{n}
\end{array}\right.
\end{equation}
with high probability as $n \to \infty$. More precisely, one can show using the methods of~\cite{CG,Sch} (see~\cite{BMS1,Sam}), and also using those of~\cite{BMS2,ST}, that (with high probability) the maximum-size sum-free subsets of $A$ contain only $o(pn)$ even numbers. Moreover, a corresponding result holds for any even-order abelian group. This fact will be a key tool in the proof below.

We will be interested in the following more precise question, which was first studied by Balogh, Morris and Samotij~\cite{BMS1}. Given an even-order abelian group $G$, note that the maximum-size sum-free subsets of $G$ are exactly the odd cosets of subgroups of index 2, and that a $p$-random subset $A \subset G$ has a sum-free subset of (expected) size
\begin{equation}\label{eq:lowerbound}
\max\big\{ |A \cap \O| \,:\, \O \textup{ is the odd coset of a subgroup of index 2} \big\} \, \ge \, \left( \frac{1}{2} + o(1) \right) p |G|.
\end{equation}
For which functions $p = p(n)$ is it true that, with high probability, the size of the largest sum-free subset of $A$ is equal to the left-hand side of~\eqref{eq:lowerbound}? In other words, for which densities does the \emph{exact} extremal result in $G$ transfer to the sparse random setting? Proving such exact extremal results is often extremely difficult; for example, the threshold for Mantel's theorem was determined only very recently by DeMarco and Kahn~\cite{DMK}. Nevertheless, it was shown in~\cite{BMS1} that the threshold for this property is $\big( \frac{\log n}{n} \big)^{1/2}$ for every even-order\footnote{In fact Theorem~1.1 of~\cite{BMS1} is more general: it determines the threshold for any abelian group whose order has a (fixed) prime factor $q$ with $q \equiv 2 \pmod 3$. Here, as before, we set $|G| = qn$.} abelian group, and moreover that there is a sharp threshold at $\big( \frac{\log n}{3n} \big)^{1/2}$ in the group $\Z_{2n}$. In other words, writing $\SF(A)$ for the collection of maximum-size sum-free subsets of $A$, and $\O_{2n}$ for the set of odd numbers in $\Z_{2n}$, they proved that for every $\eps>0$,
$$\Pr\Big( \SF\big( (\Z_{2n})_p \big) = \big\{ (\Z_{2n})_p \cap \O_{2n} \big\} \Big) \to \left\{
\begin{array} {c@{\quad \textup{if} \quad}l}
0 & p \le \big(1 - \eps \big) \sqrt{ \frac{\log n}{3n} } \\[+1ex]
1 & p \ge \big(1 + \eps \big) \sqrt{ \frac{\log n}{3n} }
\end{array}\right.$$
as $n \to \infty$. For more on the general theory of the existence of (sharp) thresholds, we refer the reader to~\cite{BT,Friedgut,Hatami}, and to~\cite{FRRT} for an example involving monochromatic triangles.

Since Balogh, Morris and Samotij~\cite{BMS1} were able to prove such a sharp threshold for the group $\Z_{2n}$, but only a weaker threshold result for other even-order abelian groups, it   is natural to ask whether one can also obtain a more precise result in the general setting. In this paper we answer this question in the affirmative, by determining the sharp threshold for every even-order abelian group. In order to state our main theorem, we shall need the following function, which determines the location of the sharp threshold.

\begin{defn}
Given an abelian group $G$ with $|G| = 2n$, let $r(G)$ denote the number of elements $x \in G$ such that $x = -x$, and set
$$\alpha(G) \, := \, \frac{\log r(G)}{\log n} \qquad \textup{and} \qquad \beta(G) \, := \, \frac{r(G)}{n}.$$
Now, given $\delta > 0$, define $\lambda^{(\delta)}(G)$ as follows:
 $$\lambda^{(\delta)}(G) \, := \, \left\{
\begin{array} {c@{\quad \textup{if} \quad}l}
1/3 & \alpha(G) \le 5/6 \\[+1ex]
\alpha(G) - 1/2 & \alpha(G) > 5/6 \textup{ and } \beta(G) < \delta\\[+1ex]
2 / \big( 4 - \beta(G) \big) & \beta(G) \ge \delta.
\end{array}\right.$$
\end{defn}

We encourage the reader to think of $\delta$ as a function going to zero slowly, and $n$ as a function going to infinity much faster. The following theorem is our main result.

\begin{thm}\label{thm:even:abelian}
For every $\eps > 0$, and every sufficiently small $0 < \delta < \delta_0(\eps)$, there exists $n_0(\eps,\delta) \in \N$ such that the following holds for every $n \ge n_0(\eps,\delta)$. Let $G$ be an abelian group of order $2n$, and let $p \in (0,1)$ with $p \ge (\log n)^2 / n$. If $A$ is a $p$-random subset of $G$, then
  $$\Pr\Big( A \cap \O \in \SF(A) \textup{ for some } \O \in \SF(G) \Big) = \left\{
\begin{array} {c@{\quad \textup{if} \quad}l}
o(1) & p \le \big(1 - \eps \big) \sqrt{ \lambda^{(\delta)}(G) \frac{\log n}{n} } \\[+1ex]
1 + o(1) & p \ge \big(1 + \eps \big) \sqrt{ \lambda^{(\delta)}(G) \frac{\log n}{n} }.
\end{array}\right.$$
\end{thm}

Here, as usual, $o(1)$ denotes a function that tends to zero as $n \to \infty$. We shall refer to the two bounds as the 0- and 1-statements respectively.

The proof of Theorem~\ref{thm:even:abelian} uses the method of~\cite{BMS1}, but we will require several substantial new ideas in order to overcome various obstacles which do not occur in the case $G = \Z_{2n}$. Many of these arise from the fact that $\SF(G)$ can be quite large (as big as $|G|$ in the case of the hypercube), which means that we must obtain much stronger bounds than in~\cite{BMS1} if we wish to apply the union bound. For the 0-statement we shall do this using a recent concentration inequality of Warnke~\cite{Warnke}, which allows us to deduce for almost all $\O \in \SF(G)$ that, with very high probability, the set $A \cap \O$ is not a maximal sum-free set. For the 1-statement, however, such a straightforward strategy is not feasible, since the threshold for the event that $A \cap \O$ is maximal for \emph{every} odd coset $\O \in \SF(G)$ is not given by $\lambda^{(\delta)}(G)$.

In order to avoid this problem, we need to show that $A \cap \O$ is a maximal sum-free set for each $\O \in \SF(G)$ such that $|A \cap \O|$ is maximal. Unfortunately, conditioning on the size of $A \cap \O$ introduces significant dependence between odd cosets, and our first attempts to prove the 1-statement failed as a consequence. We resolve this issue by fixing the \emph{number} of elements of $A$ (i.e., coupling with the hypergeometric distribution), which essentially eliminates the positive correlation between the quantities $|A \cap \O|$ for different cosets.

\pagebreak

A third issue involves the analysis of the Cayley graphs $\G_S$ for each $S \subset \E$, where $\E$ is a subgroup of index~$2$, $V(\G_S) = \O$ (the corresponding odd coset) and $xy \in E(\G_S)$ if either $x + y \in S$ or $x - y \in S$. Although counting the edges in these graphs precisely is not entirely trivial, we are fortunate that we can absorb most of the resulting mess into an error term. However, we still need to do some rather careful (and delicate) counting of the number of sets~$S$ that contain a given number of edges of $\HH_W$, the Cayley graph of the set $W = \{ a + a : a \in \O \}$, since this controls the size of $e(\G_S)$, see Section~\ref{sec:edge:counts}.

The remainder of the paper is organized as follows. In Section~\ref{sec:preliminaries}, we recall the structural version of~\eqref{eq:CGS} for even-order abelian groups from~\cite{BMS1}, and collect some probabilistic tools and simple group-theoretic facts that will be needed later. In Section~\ref{sec:edge:counts} we analyse the Cayley graph $\G_S$ for each set $S \subset \E$, where $\E$ is a subgroup of index 2, and count the number of such sets $S$ whose Cayley graph has fewer edges than expected. In Section~\ref{sec:zero:statement} we deduce the 0-statement from Warnke's concentration inequality (see Section~\ref{sec:prob:tools}), together with some of the more straightforward bounds from Section~\ref{sec:edge:counts}. Finally, in Section~\ref{sec:one:statement} we prove the 1-statement of Theorem~\ref{thm:even:abelian} using the method of~\cite{BMS1}, combined with the coupling argument and careful counting described above. We end the paper with a short Appendix, which contains a somewhat technical calculation involving the hypergeometric distribution.

\section{Preliminaries}\label{sec:preliminaries}

In this section we shall lay the groundwork necessary for the proof of our main theorem. In particular, we will recall the asymptotic stability version of Theorem~\ref{thm:even:abelian}, which was proved in~\cite{BMS1} using the method of~\cite{BMS2,CG,Sam,ST,Sch}. We will also recall the FKG inequality and the concentration inequalities of Warnke and Janson, and state some simple facts about abelian groups that will be useful in the proof.

\subsection{Sparse stability for sum-free sets}

We begin by recalling the following  theorem from~\cite{BMS1}, which determines the asymptotic structure of the maximum-size sum-free subsets in a $p$-random subset of an even-order abelian group. The theorem follows by either the method of Conlon and Gowers~\cite{CG}, or that of Schacht~\cite{Sch} (as modified by Samotij~\cite{Sam}), or that of Balogh, Morris and Samotij~\cite{BMS2} and Saxton and Thomason~\cite{ST}, in each case using results of Lev, {\L}uczak, and Schoen~\cite{LLS} and Green and Ruzsa~\cite{GR}. We refer the reader to Sections~2 and~3 of~\cite{BMS1} for the details.

\begin{thm}[Theorem~3.1 of~\cite{BMS1}]
\label{thm:approx:stability}
For every sufficiently small $\delta > 0$, there exists a constant $C = C(\delta) > 0$ such that the following holds. Let $G$ be an abelian group of order $2n$. If
  \[
  p \, \ge \, \frac{C}{\sqrt{n}},
  \]
  then, with high probability as $n \to \infty$, for every sum-free subset $B \subset G_p$ with
  \[
  |B| \, \ge \, \left( \frac{1}{2} - \delta \right) p |G|,
  \]
  there is an $\O \in \SF(G)$ such that $|B \setminus \O| \le \delta pn$.
\end{thm}

We remark that the probability of failure in Theorem~\ref{thm:approx:stability} is exponentially small in $pn$.

\subsection{Probabilistic tools}\label{sec:prob:tools}

Our main tool for the 0-statement will be the following concentration inequality, recently proved\footnote{In fact the theorem stated here is only a special case of Warnke's inequality; for the sake of simplicity, we have chosen to state only the version we need.} by Warnke~\cite[Theorem~4]{Warnke}.

\begin{Warnke}\label{thm:Warnke}
Given $N \in \N$, let $\Gamma \subset \{0,1\}^N$ be an event and $f \colon \{0, 1\}^N \to \mathbb{R}$ be a function. Let $p > 0$ and $X = (X_1, \ldots, X_N)$, where $X_k \in \{0,1\}$ and $\Pr(X_k = 1) = p$ for each $k \in [N]$, all independently, and set $\mu = \Ex\big[ f(X) \big]$. Suppose that, for some $c,d > 0$,
$$|f(x) - f(y)| \,\le\,
\begin{cases}
\;c & \textup{if } x \in \Gamma, \\
\;d & \textup{otherwise}
\end{cases}
$$
whenever $x,y \in \{0,1\}^N$ with $|x - y| = 1$, and let $\gamma \in (0,1)$.

There exists an event $\mathcal{B} = \mathcal{B}(\Gamma, \gamma) \subset \{0,1\}^N$, with $\neg \mathcal{B} \subset \Gamma$, such that
$$\Pr\big( X \in \mathcal{B} \big) \, \le \, \frac{N}{\gamma} \cdot \Pr \big( X \not\in \Gamma \big),$$
and moreover, setting $C = c + \gamma\big( d - c \big)$, we have
$$\Pr\big( f(X) \le \mu - t \text{ and } \neg \mathcal{B} \big) \, \le \, \exp \left(- \frac{t^2}{2C^2 pN + Ct} \right)$$
for any $t \ge 0$.
\end{Warnke}

We also recall two well-known probabilistic inequalities: Janson's inequality and the FKG inequality. We refer the reader to~\cite{AS} for various more general statements and their proofs.

\begin{Janson}\label{janson}
  Suppose that $\{B_i\}_{i \in I}$ is a family of subsets of a finite set $X$ and let $p \in [0,1]$. Let
  \[
  \mu = \sum_{i \in I} p^{|B_i|}, \quad \text{and} \quad
  \Delta = \sum_{i \sim j} p^{|B_i \cup B_j|},
  \]
  where $i \sim j$ denotes the fact that $i \neq j$ and $B_i \cap B_j \neq \emptyset$. Then,
  \[
  \Pr\big( B_i \not\subset X_p \text{ for all $i \in I$} \big) \le e^{-\mu + \Delta}.
  \]
  Furthermore, if $2c\mu \le \Delta$ with $c\le 1/4$, then
  \[
  \Pr\big( B_i \not\subset X_p \text{ for all $i \in I$} \big) \le e^{-c\mu^2/ \Delta}.
  \]
\end{Janson}

\begin{FKG}  \label{fkg}
  Suppose that $\{B_i\}_{i \in I}$ is a family of subsets of a finite set $X$ and let $p \in [0,1]$. Then
$$\Pr\big( B_i \not\subset X_p \text{ for all $i \in I$} \big) \, \ge \, \prod_{i \in I} \Pr\big( B_i \not\subset X_p \big).$$
\end{FKG}

\subsection{Group-theoretic facts}

In order to avoid repetition, we shall assume throughout the paper that $G$ is a finite abelian group of order $2n$. Given a subset $X \subset G$, we write
\begin{itemize}
\item $R(X)$ for the collection of elements $x \in X$ for which $x=-x$, and $r(X) = |R(X)|$.\smallskip
\item $m(X)$ for number of two-element subsets of $X$ that are of the form $\{x,-x\}$.
\end{itemize}
We will need a few basic facts about finite abelian groups. The first one is well-known.

\begin{fact}\label{fa:classification}
There exist integers $1 \le a_1 \le \dots \le a_k$ and an odd-order group $J$ such that
$$G \, \cong \, \mathbb{Z}_{2^{a_1}} \oplus \dots \oplus \mathbb{Z}_{2^{a_k}} \oplus J.$$
\end{fact}

The second fact we need is a characterization of the index $2$ subgroups of $G$.

\begin{fact}\label{mainfact}
Let $I \subset \{1, \ldots, k\}$. Writing $x \in G$ as $(x_1, \ldots, x_k, y)$ via the isomorphism of Fact~\ref{fa:classification}, the subgroup $H_I = \big\{ x \in G \, : \, \sum_{i \in I} x_i \equiv 0 \; (\textup{mod} \, 2) \big\}$ is isomorphic to $$\mathbb{Z}_{2^{a_1}} \oplus \dots \oplus \mathbb{Z}_{2^{a_i-1}} \oplus \dots \oplus \mathbb{Z}_{2^{a_k}} \oplus J,$$
where $i = \min I$. Moreover, every subgroup of $G$ of index $2$ is equal to $H_I$ for some $I \neq \emptyset$.
\end{fact}

\begin{proof}
Without loss of generality, assume that $J = \{0\}$ (and thus omit the last coordinate of elements of $G$) and $I = \{1, \ldots, k\}$. Then the image of the (injective) homomorphism
\begin{align*}
  f \colon H_I &\to \mathbb{Z}_{2^{a_1}} \oplus \dots \oplus \mathbb{Z}_{2^{a_k}} \\
  (x_1, \ldots, x_k) &\mapsto (x_1 + \ldots + x_k, x_2, \ldots, x_k)
\end{align*}
consists of the elements of $G$ whose first coordinate is even. Observe that the addition above is well-defined because there is a natural projection from $\mathbb{Z}_{2^{a_i}}$ to $\mathbb{Z}_{2^{a_1}}$ for any $1 \le i \le k$.

Conversely, given a subgroup $H$ of index $2$, observe that $\mathbbm{1}_{H^c}$ is a homomorphism onto $\mathbb{Z}_2$, which implies that $\mathbbm{1}_{H^c} (x_1, \ldots, x_k) \equiv \sum_{i=1}^k x_i \mathbbm{1}_{H^c}(e_i) \equiv \sum_{i : e_i \notin {H}} x_i \; (\textup{mod} \, 2)$, and thus $H = \big\{x \in G \, : \, \sum_{i : e_i \notin H} x_i \equiv 0 \; (\textup{mod} \, 2) \big\}$.
\end{proof}

Note that Fact~\ref{mainfact} implies that $G$ has exactly $r(G) - 1$ index 2 subgroups. Finally, we make a simple but useful observation.

\begin{fact}
  For any subgroup $H$ of $G$ of index $2$, either $r(H) = r(G)$ or $r(H) = r(G\setminus H)$.
\end{fact}

\begin{proof}
  For any $x \in R(G \setminus H)$, $y \mapsto y + x$ is a bijection between $R(H)$ and $R(G\setminus H)$.
\end{proof}

\section{Edge counts in Cayley graphs}\label{sec:edge:counts}

In order to bound the probability of the event ``$A \cap \O \in \SF(A)$" for some fixed maximum-size sum-free set $\O \in \SF(G)$ and its corresponding set of evens $\E = G \setminus \O$, we will need to consider events of the form
$$\textup{``$\big( (A \cap \O) \cup S \big) \setminus T$ is sum-free"}$$
where $S \subset A \cap \E$, $T \subset A \cap \O$ and $|S| \ge |T|$. This event is contained in the event that $(A \cap \O) \setminus T$ is an independent set in the Cayley graph $\G_S$, defined below, and to bound its probability we will need to analyse carefully the number of edges in this Cayley graph for each such set $S$ of evens. In particular, there may be an exceptional collection of sets $S$ with too few edges for our purposes (that is, for our application of the union bound over all sets $S$), and we will need to bound the size of this collection.

Let us begin by stating precisely the main results we will prove in this section. We fix throughout an arbitrary $\eps > 0$, a sufficiently small $\delta > 0$ and a sufficiently large $n \in \N$.\footnote{We think of $\delta$ as a function of $n$ which tends to zero sufficiently slowly as $n \to \infty$.} We also fix an abelian group $G$ of order $2n$, an odd coset $\O \in \SF(G)$, and its corresponding set of evens $\E = G \setminus \O$, which is a subgroup of $G$ of index 2. For each set $S \subset \E$, we define the Cayley graph $\G_S$ of $S$ to have vertex set $\O$ and edge set
$$E(\G_S) \, = \, \bigg\{ \{y,z\} \in \binom{\O}{2} \,:\, y+z \in S \textup{ or } y-z \in S \bigg\},$$
where (for simplicity) we do not permit $\G_S$ to have loops. Recall that we write $r(X)$ for the number of order 2 elements in $X \subset G$, and $m(X)$ for the number of pairs $\{x,-x\} \subset X$.

We will prove the following propositions.

\begin{prop}\label{prop:edge:counts}
Let $k \in \N$. For every $0 \not\in S \subset \E$ with $|S| = k$ and $m(S) = 0$, we have
$$ \left( \ds\frac{3k - r(S)}{2} \right) n - O\big( r(G) \cdot k^2 \big) \le e(\G_S) \le \left( \ds\frac{3k - r(S)}{2} \right) n.$$
Moreover, if $r(G) \le \delta n$ and $4\delta \le a \le 1$, then there are at most $\big(6/\delta^2\big)^k\big( n / k \big)^{k - (a/2 - \delta) k}$ sets $0 \not\in S \subset \E$ with
$$e(\G_S) \le \left( \ds\frac{3k - r(S)}{2} - ak \right) n$$
such that $|S| = k$ and $m(S) = 0$.
\end{prop}

When $r(G) \ge \delta n$ the edge counts are slightly different.

\begin{prop}\label{prop:edge:counts:bigR}
If $r(G) \ge \delta n$, then, for every $k \in \N$ and $0 \le s \le k$, there are at most $\big(12/\delta\big)^k \big( n / k \big)^s$ sets $0 \not\in S \subset \E$ with
\begin{equation}\label{eq:prop:edgecounts:bigR}
e(\G_S) < \big( s + 1 \big) \bigg( n - \frac{r(\O)}{2} \bigg)
\end{equation}
such that $|S| = k$ and $m(S) = 0$.
\end{prop}

In order to prove Propositions~\ref{prop:edge:counts} and~\ref{prop:edge:counts:bigR}, we will first count edges in $\G_x = \G_{\{x\}}$ for each $x \in \E$, and then study the intersections between these graphs. These will depend on the parameter $r(S)$, as the reader can see from the statement. However, they will also depend on the intersection of $S$ with the set $$W = \{a+a : a \in \O\},$$ and with its Cayley graph. We will use several times the fact that $|W| = n / r(\E)$.

\subsection{Edge counts in $\G_x$}

We begin with the relatively simple task of counting the edges in the Cayley graph of a single vertex $x$. To be precise, we will prove the following lemma.

\begin{lemma}\label{lem:count:Gx}
For every $0 \neq x \in \E$,
$$e(\G_x) \, = \, n - \frac{r(\O)}{2}  - \frac{r(\E)}{2}\mathbbm{1}\big[ x \in W \big] + \left(\frac{n-r(\O)}{2}\right)\mathbbm{1}\big[ x \notin R(G) \big],$$
and $\Delta(\G_x) \le 3$.
\end{lemma}

\begin{proof}
Let us denote by $\G_x^+$ the edges of the form $x = y + z$, and by $\G_x^-$ the edges of the form $x = y - z$, so $\G_x = \G_x^+ \cup \G_x^-$. Note first that the graph $\G_x^-$ has a very simple structure, since every vertex has degree either one or two. More precisely, if $x \not\in R(G)$ then it is a union of cycles, and so $e( \G_x^- ) = n$; if $x \in R(G)$ then it is a matching, and so $e( \G_x^- ) = n/2$.

In order to count the edges of $\G_x^+ \setminus \G_x^-$, let us partition the vertex set $\O$ into (up to) four parts, as follows:
\begin{itemize}
\item[$(a)$] Set $O_1 = \{ a \in \O : a+a = x\}$. If $|O_1| \neq 0$, then $x \in W$, and moreover $|O_1| = r(\E)$, since the property $a \in O_1$ is invariant under the addition of an order 2 element. Moreover $O_1$ contains no edges of $\G_x^+$, and $O_1 \cap R(\O) = \emptyset$, since $x \neq 0$.
\item[$(b)$] Set $O_2 = R(\O)$, the collection of order 2 elements in $\O$. If $x \in R(G)$ then $O_2$ induces a matching in $\G_x^+$, since $a \in R(\O)$ if and only if $b = x - a \in R(\O)$.
\item[$(c)$] Set $O_3 = \{b \in \O \setminus O_2 : x - b \in R(\O) \}$, and observe that if $x \in R(G)$ then $|O_3| = 0$ (as above), whereas if $x \not\in R(G)$ then $|O_3| = |O_2|$, since if $a \in R(\O)$ then $b = x - a \not\in R(\O)$. Moreover $\G_x^+$ contains one edge for each element of $O_3$.
\item[$(d)$] Set $O_4 = \O \setminus \big( O_1 \cup O_2 \cup O_3 \big)$, and note that $\G_x^+$ induces a perfect matching on $O_4$.
\end{itemize}
Now, observe that an edge of $\G_x^+$ is also contained in $\G_x^-$ if and only if it has an endpoint in~$R(G)$, since if $a + b = x$ then $b \in R(G)$ if and only if $a - b = x$. Therefore
$$e(\G_x) \, = \, \big( 1 + \mathbbm{1}\big[ x \not\in R(G) \big] \big) \frac{n}{2} + \frac{|O_4|}{2}$$
and
$$|O_4| \, = \, n - \mathbbm{1}\big[ x \in W \big] r(\E) - \big( 1 + \mathbbm{1}\big[ x \not\in R(G) \big]  \big) r(\O),$$
and so the lemma follows.
\end{proof}

Lemma~\ref{lem:count:Gx} has the following simple consequence,  which we shall use several times.

\begin{obs}\label{obs:edge_lower_bound} For every $0\neq x\in\E$, we have $e(\G_x) \ge \max\{n-r(G), n/2\}$. Moreover, if $0\notin S\subset \E$ satisfies $m(S) = 0$, then $e(\G_S) \ge \sum_{x \in S} e(\G_x)/2$.
\end{obs}

\begin{proof}
If $x \neq 0$, Lemma~\ref{lem:count:Gx} implies that
\begin{equation*}
e(\G_x) \ge n - \frac{r(\O)}{2} - \frac{r(\E)}{2} \mathbbm{1}[x \in W]
\end{equation*}
and, in particular, $e(\G_x) \ge n - r(G)$. In addition, either $r(\O) \le r(\E) \le n/2$ or $|W| = n/r(\E) = 1$, and so $e(\G_x) \ge n/2$. Further, when $m(S) = 0$, the set $\{x \in S : \{a, b\} \in E(G_x)\}$ contains at most two elements for any edge $\{a,b\}$.
\end{proof}

Before continuing to the proof of Proposition~\ref{prop:edge:counts}, let us note how to obtain (heuristically) the function $\lambda^{(\delta)}(G)$ from Lemma~\ref{lem:count:Gx}. We call an element $0 \ne x \in \E$ \emph{safe} if $(A \cap \O) \cup \{x\}$ is sum-free, and let $S^\E(A)$ denote the collection of safe elements in~$\E$. Note that an element $x \in \E$ is safe if\footnote{This is only true if we ignore sums of the form $x = y + y$. However, such sums will never play a significant role in any of the calculations below.} and only if $A \cap \O$ is an independent set in $\G_x$.

We need one more definition, whose slightly odd appearance will be motivated by the lemmas below.

\begin{defn}
A subgroup $\E \subset G$ is \emph{nice} if either $r(G) \le \delta n$ or $r(\O) = r(\E)$.
\end{defn}

The next lemma says that almost all index~$2$ subgroups are nice.

\begin{lemma}
$G$ has at most $2/\delta$ index $2$ subgroups that are not nice.
\end{lemma}

\begin{proof}
Clearly if $r(G) \le \delta n$ then all subgroups are nice, so let us assume $r(G) \ge \delta n$. By Fact~\ref{fa:classification}, we can write $G\cong \Z_2^k\oplus H$, where $H=\Z_{2^{a_1}} \oplus \ldots \oplus \Z_{2^{a_\ell}}\oplus J$ with $2 \le a_1 \le \cdots \le a_\ell$ and $|J|$ odd. Since $r(G)=2^{k+\ell}$ and $|G|\ge 2^{k+2\ell}$, Fact~\ref{mainfact} implies that there are at most $2^{\ell}\le 2 / \delta$ subgroups $\E \subset G$ of index $2$ that are not isomorphic to $\Z_2^{k-1} \oplus H$. But if $\E \cong \Z_2^{k-1} \oplus H$, then $r(\O)=r(\E)$, as required.
\end{proof}

We now prove the following bound on the expected number of safe elements, which we will use in the proof of the 0-statement of Theorem~\ref{thm:even:abelian}.

\begin{lemma}\label{lem:expected:safe}
If $\frac{\log n}{n} \ll p \le \big(1 - \eps \big) \sqrt{ \lambda^{(\delta)}(G) \frac{\log n}{n} }$ and $\E$ is nice, then
$$\Ex\big[ |S^\E(A)| \big] \, \gg \, \frac{\log n}{p}.$$
\end{lemma}

\begin{proof}
Suppose first that $r(G) \le \delta n$, and to simplify the notation let us write $\delta = o(1)$ (as noted above, we may assume that this holds as $n \to \infty$), and thus $r(G) = o(n)$. It follows from Lemma~\ref{lem:count:Gx} that
\begin{equation}\label{eq:edges:rGsmall}
e(\G_x) = \left\{
\begin{array} {c@{\quad \textup{if} \quad}l}
n + o(n) & x \in R(G) \\[+1ex]
3n/2 + o(n) & x \not\in R(G).
\end{array}\right.
\end{equation}
Now, by the FKG inequality, the expected number of safe elements $x \in \E$ is at least
$$\Ex[|S^\E(A)|] \, \ge \, \sum_{x \in \E} \big( 1 - p^2 \big)^{e(\G_x)} \ge \, r(\E) e^{-p^2 ( n + o(n) )} + \big( n - r(\E) \big) e^{-p^2 ( 3n/2 + o(n) )} \, \gg \, \frac{\log n}{p}.$$
To see the final step, it suffices to check that the claimed inequality holds at the endpoints of the claimed range of $p$, since $x e^{-cx^2}$ is unimodal. At the lower end this is immediate; at the upper end, note that $e^{-p^2n} \ge n^{(1-\eps)^2 \lambda^{(\delta)}(G)}$ and $r(\E) = n^{\alpha(G) + o(1)}$, and that
$$\max\bigg\{\alpha(G) - \lambda^{(\delta)}(G) , \, 1 - \frac{3\lambda^{(\delta)}(G)}{2} \bigg\} \, = \, \frac{1}{2},$$
since $\lambda^{(\delta)}(G) = \max\big\{ 1/3, \alpha(G) - 1/2 \big\}$.

When $r(G) \ge \delta n$, the (asymptotic) number of edges of $\G_x$ depends on both whether $x \in R(G)$ and whether $x \in W$. Indeed, the following table summarizes the content of Lemma~\ref{lem:count:Gx}.
\begin{table}[H]
\begin{tabular}{| c | c | c |} \cline{2-3} \multicolumn{1}{c|}{}
 &&\\[-2ex]
\multicolumn{1}{c|}{} & $x \in R(G)$ & $x \notin R(G)$ \\[+0.3ex] \hline
   &&\\[-2.2ex]
     $x \in W$ & $\displaystyle n - \frac{r(\O)}{2} - \frac{r(\E)}{2}$ & $\displaystyle \frac{3n}{2} - r(\O) - \frac{r(\E)}{2}$ \\[+1.5ex]  \hline
      &&\\[-2ex]
        $x \not\in W$ & $\displaystyle n - \frac{r(\O)}{2}$ & $\displaystyle \frac{3n}{2} - r(\O)$ \\[+1.5ex] \hline
\end{tabular}
\caption{Summary of Lemma \ref{lem:count:Gx}}
\label{table:thetable}\end{table}
\noindent Fortunately, however, $|W| = n / r(\E) = O(1/\delta)$.  We can therefore easily deduce a lower bound on $\Ex[|S^\E(A)|]$ for nice subgroups. Indeed, since $r(\O) = r(\E) = \beta(G) n/2$, and again using the unimodality of $x e^{-c x^2}$, it follows from Table \ref{table:thetable} above that
\begin{equation}\label{eq:safe:hypercube}
\Ex\big[ |S^\E(A)| \big] \, \ge \, \sum_{x \in R(\E)} \big( 1 - p^2 \big)^{e(\G_x)} = \, \Omega \Big( r(\E) e^{-p^2 ( n - r(\O)/2 )} \Big)  \, \gg \, \frac{\log n}{p},
\end{equation}
as required, where the last step follows since $1 - \big( 1 - \beta(G)/4 \big) \lambda^{(\delta)}(G) = 1/2$.
\end{proof}

\subsection{Intersections between the graphs $\G_x$ and edge counts in $\G_S$}

We now return to the proof of Proposition~\ref{prop:edge:counts}. In order to deduce the claimed bounds on $e(\G_S)$, we will need to control the size of the intersections between different graphs $\G_x$. Recall that we have fixed an odd coset $\O \in \SF(G)$, and that $W = \{a+a : a \in \O\}$. The following observation is key.

\begin{obs}\label{obs:GxGy:exists}
Let $x,y \in \E$ with $x \not\in \{y,-y\}$. If $E(\G_x) \cap E(\G_y) \neq \emptyset$, then $x+y \in W$.
\end{obs}

\begin{proof}
Suppose the edge $\{a,b\}$ lies in both $\G_x$ and $\G_y$. Then, without loss of generality, we have $a + b = x$ and $a - b = y$, and so $x + y = a + a$, as claimed.
\end{proof}

Moreover, we can bound the size of each intersection.

\begin{obs}\label{obs:GxGy:size}
$\big| E(\G_x) \cap E(\G_y) \big| \le 2 \cdot r(\E)$ for every $x,y \in \E$ with $x \not\in \{y,-y\}$.
\end{obs}

\begin{proof}
  Consider $\{a,b\}, \{c,d\} \in E(\G_x) \cap E(\G_y)$. Since $x \not\in \{y,-y\}$, we may assume that $\{a+b,a-b\}=\{x,y\}=\{c+d,c-d\}$. It follows that $a + a = x + y = c + c$, and thus $c - a \in R(\E)$. Moreover $d \in \{x-c, y-c\}$, and therefore, given $\{a,b\}$, there are at most $2 \cdot r(\E)$ choices for $\{c,d\}$, as claimed.
\end{proof}

Let us denote by $\HH_W$ the graph on vertex set $\E$ with edge set $\{xy : x+y \in W \}$, and note that we have $\Delta\big( \HH_W \big) \le d$, where $d := |W| = n/r(\E)$. By Observations~\ref{obs:GxGy:exists} and~\ref{obs:GxGy:size}, we have
\begin{equation}\label{eq:EGxEGy}
\sum_{x,y \in S, \, x \ne y} \big| E(\G_x) \cap E(\G_y) \big| \, \le \, 2 \cdot r(\E) \cdot e(\HH_W[S])
\end{equation}
for every $S \subset \E$ with $m(S) = 0$. Since, by Lemma~\ref{lem:count:Gx}, we have good bounds on the sum of $e(\G_x)$ over $x \in S$, the following lemma is all we need to complete the proof of Proposition~\ref{prop:edge:counts}.

\begin{lemma}\label{lemma:choices_for_s}
For every $\delta \le a \le 1/2$, there are at most $\big(6/\delta^2\big)^k\big( n / k \big)^{k - (1 - \delta) a k}$ sets $S \subset \E$ with $|S| = k$ and
\begin{equation}\label{eq:edges:HU}
e\big( \HH_W[S] \big) \,\ge\, \frac{akn}{r(\E)}.
\end{equation}
\end{lemma}

\begin{proof}
We shall first bound the number of sequences $(v_1, \ldots, v_k) \in \E^k$ such that the set $S = \{v_1, \ldots, v_k\}$ satisfies $|S| = k$ and~\eqref{eq:edges:HU}. Given such a sequence, let us say (for each $j \in [k]$) that the vertex $v_j$ is of `low degree' if it is connected (by edges of $\HH_W$) to fewer than $\delta a d = \delta a n / r(\E)$ vertices of the set $\{v_1, \ldots, v_{j-1}\}$, and say it is of high degree otherwise.

Since $\Delta(\HH_W) \le d$, it follows from~\eqref{eq:edges:HU} that in each such sequence there must be at least $(1 - \delta)ak$ high-degree vertices, since the low-degree vertices contribute fewer than $\delta akd$ edges. Moreover, since there are at most $(j-1)d < kd$ edges of~$\HH_W$ leaving the set $\{v_1, \ldots, v_{j-1}\}$, there are at most $k / \delta a$ choices for a high-degree vertex, given the collection of vertices which have already been chosen.

Now, given a set $J \subset [k]$ of size at least $(1-\delta)ak$, corresponding to the positions of vertices which are required to have high degree, there are at most
$$\left(\frac{k}{\delta a}\right)^{|J|}n^{k - |J|}$$
possible sequences, and this value is maximised when $|J|$ is minimised. Therefore, considering all possible choices for $J$, it follows that there are at most
$$2^k  \left(\frac{k}{\delta a}\right)^{(1-\delta)ak}n^{k - (1-\delta)ak}$$
sequences with the desired properties.

Finally, note that each set appears exactly $k!$ times as a sequence, and therefore the number of sets $S \subset \E$ with $|S| = k$ satisfying~\eqref{eq:edges:HU} is at most
$$\bigg( \frac{2e}{k} \bigg)^k \left(\frac{k}{\delta a}\right)^{(1-\delta)ak} n^{k - (1-\delta)ak} \, \le \, \bigg(\frac{2e}{\delta^2}\bigg)^k\bigg( \frac{n}{k}  \bigg)^{k - (1 - \delta) a k},$$
since $a \ge \delta$, as required.
\end{proof}

We are now ready to prove the two propositions.

\begin{proof}[Proof of Proposition~\ref{prop:edge:counts}]
Let $0 \not\in S \subset \E$ with $|S| = k$ and $m(S) = 0$. By Lemma~\ref{lem:count:Gx} and~\eqref{eq:EGxEGy}, and noting that $|W| = n / r(\E)$, we have
\begin{align*}
e(\G_S) & \, \ge \, \sum_{x \in S} \bigg( n - \frac{r(\O)}{2}  - \frac{r(\E)}{2}\mathbbm{1}\big[ x \in W \big] + \left(\frac{n-r(\O)}{2}\right)\mathbbm{1}\big[ x \notin R(G) \big] \bigg) - 2 \cdot r(\E) e(\HH_W[S])\\
& \, \ge \, k \big( n - r(G) \big) + \left(\frac{n-r(\O)}{2}\right) \big( k - r(S) \big) - 2 \cdot r(\E) e(\HH_W[S])\\
& \, \ge \, \bigg( \frac{3k - r(S)}{2} \bigg) n - O\big( r(G) \cdot k^2 \big),
\end{align*}
as required, and the upper bound follows similarly. Moreover, the same calculation implies that if $e(\G_S) \le \big( \frac{3k - r(S)}{2} - ak \big) n$ and $r(G) \le \delta n$, then
$$e\big( \HH_W[S] \big) \, \ge \, \frac{(a - 3\delta/4)kn}{2 \cdot r(\E)},$$
and by Lemma~\ref{lemma:choices_for_s} there are at most $\big(6/\delta^2\big)\big( n / k \big)^{k - (a/2 - \delta) k}$ such sets $S \subset \E$ with $|S| = k$.
\end{proof}

\begin{proof}[Proof of Proposition~\ref{prop:edge:counts:bigR}]
The proof is similar to that of Lemma~\ref{lemma:choices_for_s}, but for completeness we give the details. We will count sequences $(v_1,\ldots,v_k) \in \E^k$ such that the set $S = \{v_1, \ldots, v_k\}$ satisfies $|S| = k$ and~\eqref{eq:prop:edgecounts:bigR}. Let $S_j=\{v_1,\ldots,v_j\}$, and observe that, since $m(S) = 0$, each $0 \ne x \not\in W$ that sends no edges of $\HH_W$ into $S_j$ adds at least $n - r(\O)/2$ edges to $\G_S$, by Lemma~\ref{lem:count:Gx} (see Table \ref{table:thetable}) and Observation~\ref{obs:GxGy:exists}. There are therefore at most $s$ such `bad' vertices, since $e(\G_s)<(s+1)(n-r(\O)/2)$.

Now, since $\Delta(\HH_W) \le |W| = n / r(\E) \le 2/\delta$ and $|S_j| = j < k$, it follows that there are at most $2k/\delta$ vertices in $W \cup N_{\HH_{W}}(S_j)$, and hence at most this many choices for each `good' vertex. Note that there are at most $2^k$ choices for the indices~$j$ such that $v_j$ is bad, and each set $S$ is counted $k!$ times as a sequence. Thus, the number of sets $0 \not\in S \subset \E$ with $|S| = k$ satisfying~\eqref{eq:prop:edgecounts:bigR} is at most
$$\frac{2^k}{k!} \cdot \bigg( \frac{2k}{\delta} \bigg)^{k-s} n^s \, \le \, \bigg(\frac{4e}{\delta}\bigg)^k\bigg(\frac{n}{k}\bigg)^s,$$
as claimed.
\end{proof}

\section{Proof of the 0-Statement}\label{sec:zero:statement}

In this section we will prove that if $A \subset G$ is a $p$-random set and
\begin{equation}\label{eq:pbounds:zerostatement}
\frac{\log n}{n} \, \ll \, p \le \big(1 - \eps \big) \sqrt{ \lambda^{(\delta)}(G) \frac{\log n}{n} },
\end{equation}
then $A \cap \O \not\in \SF(A)$ for every $\O \in \SF(G)$ with high probability as $n \to \infty$. The main step will be proving the following proposition.\footnote{We remark that the bound $1/n^2$ could easily be replaced by $1/n^C$ for any $C > 0$.}

\begin{prop}\label{prop:zero:statement}
For every $\eps > 0$, the following holds for every sufficiently large $n \in \N$. Let $G$ be an abelian group of order $2n$, let $\O \in \SF(G)$ and suppose that $\E = G \setminus \O$ is nice and that $p \in (0,1)$ satisfies~\eqref{eq:pbounds:zerostatement}. If $A$ is a $p$-random subset of $G$, then
$$\Pr\big( A \cap \O \in \SF(A) \big) \le \, \frac{1}{n^2}.$$
\end{prop}

Recall also that at most $O(1/\delta)$ of the index 2 subgroups of $G$ are not nice. We will use the following simple-sounding lemma to deal with these subgroups.

\begin{lemma}\label{lem:M:nice}
Let $\M$ denote the collection of odd cosets $\O \in \SF(G)$ such that $|A \cap \O|$ is maximal. Then with high probability there is an $\O \in \M$ such that $\E = G \setminus \O$ is nice.
\end{lemma}

The proof of Lemma~\ref{lem:M:nice}, although not difficult, is surprisingly technical, and so we shall postpone it to the appendix. Note that the 0-statement in Theorem~\ref{thm:even:abelian} follows from Proposition~\ref{prop:zero:statement} and Lemma~\ref{lem:M:nice} by taking a union bound over nice subgroups.

Recall that an element $x \in \E$ is called \emph{safe} if $(A \cap \O) \cup \{x\}$ is sum-free, and  that $S^\E(A)$ denotes the collection of safe elements in $\E$. We will bound the probability of the event $A \cap \O \in \SF(A)$ by the probability that there exists no safe element $x \in A \cap \E$. Since the random variable $S^\E(A)$ is independent of the set $A \cap \E$, it follows that
\begin{equation}\label{eq:zero:Chernoff}
\Pr\bigg( \Big( A \cap \O \in \SF(A) \Big) \cap \bigg( |S^\E(A)| \ge \frac{3 \log n}{p} \bigg) \bigg) \, \le \, \big( 1 - p \big)^{(3 \log n) / p} \, \le \, \frac{1}{n^3},
\end{equation}
and so it is enough to consider the event that $|S^\E(A)| \le (3 \log n) / p$.

We will bound the probability of this event using Warnke's concentration inequality, which was stated in Section~\ref{sec:prob:tools}. The first step -- showing that $|S^\E(A)|$ has large expected value -- was already carried out in the previous section. Indeed, we have

\begin{equation}\label{eq:safe:logn}
\Ex\big[ |S^\E(A)| \big] \, \gg \, \frac{\log n}{p}
\end{equation}
whenever $p \in (0,1)$ satisfies~\eqref{eq:pbounds:zerostatement}, by Lemma~\ref{lem:expected:safe}. Our main task will be to prove the following lemma, which shows that $|S^\E(A)|$ is concentrated around its expected value.

\begin{lemma}\label{lem:Xconcentration}
If $p \in (0,1)$ satisfies~\eqref{eq:pbounds:zerostatement}, then
$$\Pr\bigg( |S^\E(A)| \le \, \frac{\Ex\big[ |S^\E(A)| \big]}{2} \bigg) \le \, \frac{1}{n^3}.$$
\end{lemma}

We will prove Lemma~\ref{lem:Xconcentration} by applying Warnke's inequality to the function $A \mapsto |S^\E(A)|$. In order to do so, we need to define an event $\Gamma \subset \P(\O)$, and prove the `typical Lipschitz condition'
\begin{equation}\label{eq:typLips}
\big| |S^\E(A)| - |S^\E(B)| \big| \, \le \, \left\{
\begin{array} {c@{\quad}l}
c(\E,p) := n^{-(1/4 + \delta)} \cdot \Ex\big[ S^\E(A) \big] & \textup{if } A \in \Gamma, \\
n & \textup{otherwise}
\end{array}\right.
\end{equation}
for every $A,B \subset \O$ with $| A \triangle B | = 1$ (note that $c(\E, p) \gg 1$, by~\eqref{eq:safe:logn}). We define the event $\Gamma$ so that~\eqref{eq:typLips} holds by definition:
\begin{equation}\label{def:gamma}
\Gamma \, := \, \Big\{ A \subset \O \,:\, \max\big\{ \big| |S^\E(A)| - |S^\E(B)| \big| : | A \triangle B | = 1 \big\} \le c(\E,p) \Big\}.
\end{equation}
We would like to show that $\Pr \big( A \not\in \Gamma \big) \le n^{-5}$, since this will imply the desired upper bound on the probability of the event $\B$ given by Warnke's inequality.

The main technical step in the proof of Lemma~\ref{lem:Xconcentration} is proving such a bound on the probability that $A \not\in \Gamma$.  To do so, note first that if $A \notin \Gamma$ then there exists $u \in \O$ such that $\big||S^\E(A)|-|S^\E(A \Delta \{u\})|\big|>c(\E,p)$. Let $\Gamma^c(u)$ be the set of choices of $A$ for which this property holds, so that $\Gamma^c = \bigcup_{u \in \O}\Gamma^c(u)$, and note that, by symmetry,\footnote{Indeed, if $B = A \Delta \{u\}$ then $A \in \Gamma^c(u) \Leftrightarrow B \in \Gamma^c(u) \Leftrightarrow \big| |S^\E(A)| - |S^\E(B)| \big| > c(\E,p)$.}
\begin{equation}\label{eq:xnotinA}
\Pr\big( A \in \Gamma^c(u) \,\big|\, u \in A \big) \, = \, \Pr\big( A \in \Gamma^c(u) \,\big|\, u \not\in A \big).
\end{equation}
We will bound $\Pr\big(A \in \Gamma^c(u) \big)$ for each fixed $u \in \O$, and then sum over $u$.

Motivated by~\eqref{eq:xnotinA}, let us fix $u \in \O$, assume that $u \not\in A$, and write
$$Y_u^\E(A) \, = \, S^\E(A) \setminus S^\E(A \cup \{u\}).$$
Observe that $A \in \Gamma^c(u)$ if and only if $|Y_u^\E(A)| > c(\E,p)$. We will prove the following lemma.

\begin{lemma}\label{lem:prob:gamma}
For every $k$ satisfying $25 < k \le \sqrt{1/\delta}$,
$$\Pr\big( A \not\in \Gamma \big) \le \, c(\E,p)^{-k} \sum_{u \in \O} \Ex\left[ \big| Y_u^\E(A) \big|^k \right] \ll \, \frac{1}{n^5}$$
as $n \to \infty$.
\end{lemma}

Note that the first inequality follows from the comments above and Markov's inequality. The intuition behind the second inequality is based on our expectation that $|Y_u^\E(A)| =\Theta(p\big|S^\E(A)|\big)$, and that the events $\big\{ z \in Y_u^\E(A) : z \in \E \big\}$ are more or less independent of one another. We expect $|Y_u^\E(A)|$ to take roughly this value since $Y_u^\E(A) \subset S^\E(A)$, and moreover for each $z \in Y_u^\E(A)$ there is a $v \in \O$ with $uv \in E(\G_z)$ such that $v \in A$.

In order to make this argument precise, the following notion will be crucial. Fix $u \in \O$, and say that a set $0 \ne Z \subset \E$ is \emph{covered} by $Y \subset \O$ if for each $z \in Z$ there is a $y \in Y$ such that $uy \in E(\G_z)$. Say that $Z$ is \emph{cover-maximal} if $|Y| \ge |Z|$ for every set $Y$ that covers $Z$, and for each $Z \subset \E$ choose a maximum-size cover-maximal subset $g(Z) \subset Z$. Note that since any singleton in $Z$ is cover-maximal, $g(Z)$ is non-empty. The following lemma is key.

\begin{lemma}\label{lem:fewcovers}
For each $Z \subset \E$, there are at most $12^{|Z|}$ sets $Z' \subset \E$ such that $g(Z') = Z$.
\end{lemma}

\begin{proof}
Consider a set $Z' \subset \E$ such that $g(Z') = Z$. Then for any $z \in Z' \setminus Z$, there must exist some set $Y \subset \O$ of size $|Z|$ that covers $Z \cup \{z\}$ (and hence also covers $Z$), otherwise the set $Z \cup \{z\}$ contradicts the maximality in the definition of $g(Z')$.

We claim that there are at most $3^{|Z|}$ sets $Y \subset \O$ of size $|Z|$ covering $Z$. Indeed, since $Z$ is cover-maximal, $Y$ must contain exactly one element of $N_{\G_z}(u)$ for each $z \in Z$, and these neighbourhoods must be disjoint. Since $\Delta(\G_z) \le 3$, it follows that we have at most $3^{|Y|} = 3^{|Z|}$ choices for $Y$. But each such set $Y$ covers at most $3|Z|$ elements (since each is in $(Y \pm u) \cup (u - Y)$), and each $z \in Z' \setminus Z$ must be covered by some such $Y$, by the comments above. We therefore have at most $3^{|Z|} \cdot 2^{2|Z|} = 12^{|Z|}$ possible pre-images of $Z$, as claimed.
\end{proof}

We also need the following simple observation, which follows easily from the definition.

\begin{obs}\label{obs:nopairs:coverminimal}
If $Z$ is cover-maximal and $\{a,-a\} \subset Z$, then $a = -a$.
\end{obs}

\begin{proof}
The element $u + a \in \O$ covers both $a$ and $-a$, and so if $\{a,-a\} \subset Z$ and $a \neq -a$ then there exists a set $Y$ with $|Y| \le |Z| - 1$ which covers $Z$.
\end{proof}

We are ready to prove Lemma~\ref{lem:prob:gamma}.

\begin{proof}[Proof of Lemma~\ref{lem:prob:gamma}]
Consider the family $\M_k$ of non-empty cover-maximal sets $Z \subset \E$ with $|Z|=k$, and note that if $Z' \subset Z$, then trivially
$$\Pr\big( Z' \subset Y_u^\E(A) \big) \ge \, \Pr\big( Z \subset Y_u^\E(A) \big).$$
Thus, by Lemma~\ref{lem:fewcovers}, we have
$$\sum_{Z \le |k|} \Pr\big( Z \subset Y_u^\E(A) \big) \, \le \, 12^k \sum_{\ell=1}^k\sum_{Z \in \M_{\ell}} \Pr\Big( \big( | A \cap N_{\G_Z}(u) | \ge |Z| \big) \cap \big( Z \subset S^\E(A) \big) \Big),$$
since each set $Z$ contains a non-empty cover-maximal set $g(Z)$, and each such set is counted at most $12^k$ times. Now, since $|N_{\G_Z}(u)| \le 3|Z|$, the right-hand side is at most
\begin{equation}\label{eq:bound:sumoverM}
12^k \sum_{\ell = 1}^k \sum_{Z \in \M_{\ell}} 2^{3\ell} p^\ell \cdot \Pr\big( Z \subset S^\E(A) \big),
\end{equation}
by the FKG inequality, since $\big\{ Z \subset S^\E(A) \big\}$ is decreasing in $A$, whereas $\big\{ | A \cap N_{\G_Z}(u)| \ge |Z| \big\}$ is clearly increasing.

We will apply Janson's inequality to bound $\Pr\big( Z \subset S^\E(A) \big)$ for each $Z \in \M_{\ell}$. Note that $m(Z) = 0$, by Observation~\ref{obs:nopairs:coverminimal}, and that $Z \subset S^\E(A)$ implies that $A \cap \O$ is an independent set in $\G_Z$, and suppose first that $r(G) \le \delta n$. Then,
$$\mu \,:=\, p^2 e(\G_Z) \,\ge\, \bigg( p^2 \sum_{z \in Z} e(\G_z) \bigg) - O\big( \delta \ell^2 p^2 n \big) \quad \textup{and} \quad \Delta \,:=\, p^3 \sum_{v \in \O} \binom{d_{\G_Z}(v)}{2} \,=\, O\big( \ell^2 p^3 n \big),$$
since $\big| E(\G_y) \cap E(\G_z) \big| \le 2 \cdot r(\E) = O(\delta n)$ for every $y,z \in Z$ by Observation~\ref{obs:GxGy:size}. Therefore, since $e(\G_z) \ge n/2$ for every $0 \ne z \in \E$ by Observation~\ref{obs:edge_lower_bound}, and $\ell \le k\le 1/\sqrt{\delta}$, it follows by Janson's inequality and~\eqref{eq:pbounds:zerostatement} that
$$\Pr\big( Z \subset S^\E(A) \big) \, \le \, n^{O(\delta \ell^2)} \exp\bigg( - p^2 \sum_{z \in Z} e(\G_z) \bigg)  \, = \, n^{O(\delta \ell^2)} \prod_{z \in Z} \big( 1 - p^2 \big)^{e(\G_z)},$$
since $1 - p^2 \ge e^{- p^2 - p^4}$ when $p$ is sufficiently small, and $p^4 e(\G_z) = o(1)$. Thus
\begin{align}
& \sum_{Z \in \M_\ell} \Pr\big( Z \subset S^\E(A) \big) \, \le \, n^{O(\delta \ell^2)} \sum_{Z \in \M_\ell}  \prod_{z \in Z} \big( 1 - p^2 \big)^{e(\G_z)} \nonumber\\
& \hspace{3.5cm} \, \le \, n^{O(\delta \ell^2)} \bigg( \sum_{z \in \E} \big( 1 - p^2 \big)^{e(\G_z)} \bigg)^\ell \, \le \, n^{O(\delta \ell^2)} \cdot \Ex\big[ |S^\E(A)| \big]^\ell, \label{eq:expectedsafe:rsmall}
\end{align}
where the final inequality follows by the FKG inequality.

On the other hand, if $r(G) \ge \delta n$ then, by  Proposition~\ref{prop:edge:counts:bigR}, there are at most $n^{s + o(1)}$ sets $Z \subset \E$ with $|Z| = \ell$, $m(Z) = 0$ and
$$ s \bigg( n - \frac{r(\O)}{2} \bigg) \le e(\G_Z) < (s+1) \bigg( n - \frac{r(\O)}{2} \bigg).$$
Thus, applying Janson's inequality as before, we obtain\footnote{When $s = \ell$, we trivially bound the number of sets $Z$ such that $e(\G_Z) \ge \ell \left(n - \frac{r(\O)}{2}\right)$ by $n^\ell$.}
\begin{align}
& \sum_{Z \in \M_\ell} \Pr\big( Z \subset S^\E(A) \big) \, \le \, n^{o(1)} \sum_{s = 1}^{\ell} \Big( n \big( 1 - p^2 \big)^{n - r(\O)/2} \Big)^s \nonumber\\
& \hspace{3.5cm} \, = \, n^{o(1)} \Big( 1 + n \big( 1 - p^2 \big)^{n - r(\O)/2} \Big)^\ell \, \le \, n^{o(1)} \cdot \Ex\big[ |S^\E(A)| \big]^\ell, \label{eq:expectedsafe:rbig}
\end{align}
by~\eqref{eq:safe:hypercube}. Combining~\eqref{eq:bound:sumoverM},~\eqref{eq:expectedsafe:rsmall} and~\eqref{eq:expectedsafe:rbig}, it follows that
$$\Ex\left[ \big| Y_u^\E(A) \big|^k \right] \, \le \, n^{o(1)} \cdot \sum_{\ell = 1}^k \sum_{Z \in \M_\ell} p^\ell \cdot \Pr\big( Z \subset S^\E(A) \big) \, \le \, n^{O(\delta k^2)} \Big( p \cdot \Ex\big[ |S^\E(A)| \big] \Big)^k,$$
and the lemma follows, since $c(\E,p)^{-1} \cdot p \cdot \Ex\big[ |S^\E(A)| \big] \ll n^{-1/5 - \eps}$.
\end{proof}

It is now straightforward to deduce Lemma~\ref{lem:Xconcentration}, and hence Proposition~\ref{prop:zero:statement}.

\begin{proof}[Proof of Lemma~\ref{lem:Xconcentration}]
We apply Warnke's inequality to the function $A \mapsto |S^\E(A)|$ and the event $\Gamma$ defined in~\eqref{def:gamma}, with
$$c = c(\E,p) \gg 1, \qquad d = n, \qquad \gamma = \frac{c(\E,p)}{n} \qquad \textup{and} \qquad t = \frac{\Ex\big[ |S^\E(A)| \big]}{2}.$$
We obtain an event $\B$ such that
$$\Pr\big( A \in \mathcal{B} \big) \, \le \, \frac{n^2}{c(\E,p)} \cdot \Pr \big( A \not\in \Gamma \big) \, \ll \, \frac{1}{n^3},$$
where the last inequality follows by Lemma~\ref{lem:prob:gamma}, such that
\begin{align*}
\Pr\bigg( |S^\E(A)| \le \, \frac{\Ex\big[ |S^\E(A)| \big]}{2} \bigg) & \, \le \, \Pr\big( A \in \mathcal{B} \big) + \exp \left(- \frac{t^2}{4c(\E,p)^2 p n + 2c(\E,p) t} \right)\\
& \, \le \, \frac{o(1)}{n^3} \,+\, \exp \big( - n^\delta \big) \, \le \, \frac{1}{n^3},
\end{align*}
as required.
\end{proof}

\begin{proof}[Proof of Proposition~\ref{prop:zero:statement}]
We split the event $A \cap \O \in \SF(A)$ into two parts, depending on whether or not $|S^\E(A)| \le (3 \log n) / p$. By Lemmas~\ref{lem:expected:safe} and~\ref{lem:Xconcentration}, the probability that $|S^\E(A)| \le (3 \log n) / p$ is at most $1 / n^3$. On the other hand, by \eqref{eq:zero:Chernoff}, the probability that $A \cap \O \in \SF(A)$ and $|S^\E(A)| \ge (3 \log n) / p$ is at most $1/n^3$. Therefore
$$\Pr\big( A \cap \O \in \SF(A) \big) \le \, \Pr\bigg( |S^\E(A)| \le \frac{3 \log n}{p} \bigg) + \frac{1}{n^3} \, \le \, \frac{1}{n^2},$$
as required.
\end{proof}

The 0-statement now follows immediately.

\begin{proof}[Proof of the $0$-statement in Theorem~\ref{thm:even:abelian}]
Recall that an abelian group $G$ has at most $|G|$ index 2 subgroups. Thus, by Proposition~\ref{prop:zero:statement} and the union bound, it follows that with high probability $A \cap \O \not\in \SF(A)$ whenever $\E = G \setminus \O$ is nice. However, by Lemma~\ref{lem:M:nice}, with high probability there is an odd coset $\O \in \SF(G)$ such that $|A \cap \O|$ is maximal and $\E = G \setminus \O$ is nice. Hence with high probability $A \cap \O \not\in \SF(A)$ for every $\O \in \SF(G)$, as required.
\end{proof}

\section{Proof of the 1-statement}\label{sec:one:statement}

In this section we will prove that if $A \subset G$ is a $p$-random set and
$$p \ge \big(1 + \eps \big) \sqrt{ \lambda^{(\delta)}(G) \frac{\log n}{n} },$$
then every $B \in \SF(A)$ is equal to $A \cap \O$ for some $\O \in \SF(G)$, with high probability as $n \to \infty$. The proof has three steps: an application of Theorem~\ref{thm:approx:stability} to obtain an asymptotic version, an argument for a given odd coset $\O \in \SF(G)$, using the method of~\cite{BMS1} (see Lemma~\ref{lem:BMSmethod}), and a comparison with the hypergeometric distribution, which allows us to a partition the odd cosets depending on the size of $A \cap \O$ (see Lemma~\ref{le:existsb}). Recall throughout that we have already fixed an arbitrary $\eps>0$, a sufficiently small $\delta >0$ and a sufficiently large $n\in\N$.

We begin by proving the statement we will require for a given odd coset $\O \in \SF(G)$.
For each $k \in \N$, let $\B_k^\O(A)$ denote the event that there exist sets $S \subset A \cap \E$ and $T \subset A \cap \O$, with $|S| = k \ge |T|$, such that $\big( (A \cap \O ) \cup S \big)\setminus T$ is sum-free.

\begin{lemma}\label{lem:BMSmethod}
Let $G$ be an abelian group of order $2n$, and let $\O \in \SF(G)$. Suppose that
$$p \ge \big(1 + \eps \big) \sqrt{ \lambda^{(\delta)}(G) \frac{\log n}{n} },$$
and let $p_1 = (1-\delta)p$ and $p_2 = (1+\delta)p$. Set $A = A_1 \cup A_2$, where $A_1$ is a $p_1$-random subset of $\O$ and $A_2$ is a $p_2$-random subset of $\E = G \setminus \O$. Then
$$\Pr\big(\B^\O_k(A) \big) \, \le \, \max\big\{ n^{-\delta k}, e^{-\sqrt{n}} \big\}$$
for every $1 \le k \le \delta p n$.
\end{lemma}

Let us denote by $\Pr_{p^\pm} = \Pr^\O_{p^\pm}$ the probability distribution in Lemma~\ref{lem:BMSmethod}, in which each element of $\O$ is chosen (independently) with probability $(1-\delta)p$ and each element of $\E$ is chosen with probability $(1+\delta)p$. Note that the event $\B^\O_k(A)$ is increasing in $A \cap \E$ and decreasing in $A \cap \O$, so $\Pr_p\big(\B^\O_k(A) \big) \le \, \Pr_{p^\pm}\big(\B^\O_k(A) \big)$ for every $\delta \ge 0$.

\begin{proof}[Proof of Lemma~\ref{lem:BMSmethod}]
The proof of the lemma follows closely the method of Balogh, Morris and Samotij~\cite[Section~5]{BMS1}, and so we shall skip some of the details. We will bound the expected number of {\emph{good triples}} $(S,T,U)$ with the following properties:
 \begin{itemize}
  \item[$(i)$] $S \subset A \cap \E$ with $|S| = k$, \\[-2ex]
  \item[$(ii)$]  $T,U \subset A \cap \O$ are disjoint sets with $|U| \le |T| \le k$,\\[-2ex]
  \item[$(iii)$] $(A \cap \O) \setminus T$ is an independent set in $\G_S$, \\[-2ex]
  \item[$(iv)$] $T \subset N_{\G_S}(U)$.
  \end{itemize}
It was shown in~\cite[Claim~2]{BMS1} that if $\B^\O_k(A)$ holds, then there exists such a triple. Indeed, this follows by first  taking $T$ minimal, and then taking a maximal matching $M$ from $T$ to $A \setminus T$ in $\G_S$. We set $U$ equal to the set of vertices in $A \setminus T$ that are incident to $M$.

Let $Z(k,\ell,j,m,r)$ denote the number of such triples $(S,T,U)$ with $|S| = k$, $|T| = \ell$, $|U| = j$, $m(S) = m$ and $r(S) = r$. We note that by definition $2m+r\le k$, and define
$$Z_k \, := \, \sum_{\ell = 0}^k \sum_{j=0}^{\ell} \sum_{m = 0}^{k/2} \sum_{r = 0}^{k - 2m}  Z(k,\ell,j,m,r).$$
By the discussion above,
\begin{equation}\label{PatmostZ}
\Pr\big(\B^\O_k(A) \big) \, \le \, \Ex\big[ Z_k \big] \, = \, \sum_{\ell = 0}^k \sum_{j=0}^{\ell} \sum_{m = 0}^{k/2} \sum_{r = 0}^{k - 2m}   \Ex\big[ Z(k,\ell,j,m,r) \big],
\end{equation}
and therefore it will suffice to bound $\Ex[ Z(k,\ell,j,m,r) ]$ for each $k$, $\ell$, $j$, $m$ and $r$. Let $p^2 n = C \log n$, where $C \ge (1 + \eps ) \lambda^{(\delta)}(G)$. We will prove that
\begin{equation}\label{eq:Zbounds:claim}
\Ex[ Z(k,\ell,j,m,r) ] \, \le \, \left\{
\begin{array} {c@{\quad}l}
n^{-\delta k} & \textup{if} \quad k \le \delta / p \\[+1ex]
e^{-\sqrt{n}} & \textup{otherwise.}
\end{array}\right.
\end{equation}

Let us fix $k$, $\ell$, $j$, $m$ and $r$, and count the triples $(S,T,U)$ that contribute to $Z(k,\ell,j,m,r)$.  First, for each $S \subset \E$ and $\ell,j \in \N$, let $W(S,\ell,j)$ denote the number of disjoint pairs $(T,U)$ such that $T,U \subset A \cap \O$ and $T \subset N_{\G_S}(U)$, with $|T| = \ell$ and $|U| = j$. It was proved in~\cite{BMS1} that if $|S| = k$ and $0 \le j \le \ell \le k \le \delta pn$, then
$$\Ex\big[ W(S,\ell,j) \big] \, \le \, (3e^2p^2n)^k \, \ll \, \big( C \log n \big)^{2k} \, = \, n^{o(k)}$$
assuming that $C = n^{o(1)}$, as we may since the case $C \gg 1$ was already dealt with in~\cite{BMS1}.\footnote{Alternatively, we may simply carry this factor of $C^{2k}$ through the proof, and perform an easy but tedious calculation later on.}

Let $\S(k,m,r)$ denote the collection of sets $S \subset \E$ with $|S| = k$, $m(S) = m$ and $r(S) = r$. If $(S, T, U)$ is good, then no edge of the graph
$$\G_{S,T,U} \, := \, \G_S\big[ \O \setminus (T \cup U) \big]$$
has both its endpoints in $A$. Since the vertex set of $\G_{S,T,U}$ is disjoint from $S \cup T \cup U$, it follows that the events $e(\G_{S,T,U}[A]) = 0$ and $S \cup T \cup U \subset A$ are independent. Therefore,
\begin{align}
\Ex\big[ Z(k,\ell,j,m,r) \big] & \, \le \sum_{S \in \S(k,m,r)} \Pr(S \subset A) \cdot \Ex\big[ W(S,\ell,j) \big] \cdot \max_{T,U} \left\{ \Pr\Big( e\big( \G_{S,T,U}[A] \big) = 0 \Big) \right\} \nonumber\\
& \, \le \, p^k \cdot n^{o(k)} \sum_{S \in \S(k,m,r)} \max_{T,U} \left\{ \Pr\Big( e\big( \G_{S,T,U}[A] \big) = 0 \Big) \right\}, \label{eq:Zbound}
\end{align}
where the maximum is taken over all pairs $(T,U)$ as in the definition of $W(S,\ell,j)$. We will bound the probability that $A$ is an independent set in $\G_{S,T,U}$ using Janson's inequality. Indeed, let
$$\mu \,:=\, p^2 e\big( \G_{S,T,U} \big) \qquad \textup{and} \qquad \Delta \,:= \sum_{v \in \O \setminus (T \cup U)} p^3 \binom{d(v)}{2},$$
where $d(v)$ denotes the degree of $v$ in $\G_{S,T,U}$.

We break into two cases, depending on the number of elements of order 2 in $G$.

\medskip
\noindent \textbf{Case 1:} $r(G) \le \delta n$.
\medskip

For each $S \in \S(k,m,r)$ let us choose a subset $\hat{S} \subset S$ with $|\hat{S}| = k - m$, $r(\hat{S}) = r$ and $m(\hat{S}) = 0$. Applying Proposition~\ref{prop:edge:counts} to $\hat{S}$, it follows that
\begin{equation}\label{eq:eGSTU}
e(\G_{S,T,U}) \, \ge \, e(\G_{\hat{S}}) - O(k^2) \, \ge \, \left( \ds\frac{3(k-m) - r}{2} \right) n - O\big( r(G) \cdot k^2 \big),
\end{equation}
and that, for every $4\delta \le a \le 1$, the number of sets $\hat{S} \in \S(k-m,0,r)$ with
\begin{equation}\label{eq:eGhat}
e(\G_{\hat{S}}) \, \le \, \left( \ds\frac{3(k - m) - r}{2} - ak \right) n
\end{equation}
is at most $\big(6/\delta^2\big)^k \big( n / k \big)^{k - (a/2 - \delta) k}$. Moreover, for each such set $\hat{S}$ there are at most $2^k$ corresponding sets $S \in \S(k,m,r)$.
There are three sub-cases to consider: \medskip

$(a)$ Suppose first that $k \le \min\big\{ \sqrt{\delta} / p, \, \delta n / r(G) \big\}$. Then, by~\eqref{eq:eGSTU},
$$\mu \, \ge \, \left( \ds\frac{3(k-m) - r}{2} - O\big( \delta k \big) \right) p^2 n \qquad \textup{and} \qquad \Delta \, = \, O\big( k^2 p^3 n \big) = O\big( \sqrt{\delta}k p^2 n \big),$$
since $d(v) \le 3k$ for every $v \in V(\G_{S,T,U})$. Thus, by Janson's inequality, it follows that
$$\Pr\Big( e\big( \G_{S,T,U}[A] \big) = 0 \Big) \, \le \, \exp\bigg( - \left( \ds\frac{3(k-m) - r}{2} - O\big(\sqrt{\delta}k   \big) \right) p^2 n \bigg),$$
and hence, by~\eqref{eq:Zbound},
$$\Ex\big[ Z(k,\ell,j,m,r) \big] \, \le \, p^k \cdot r(\E)^r \cdot n^{k - m - r + o(k)} \cdot \exp\bigg( - \left( \ds\frac{3(k-m) - r}{2} - O\big( \sqrt{\delta}k \big) \right) p^2 n \bigg).$$
Since $p = n^{-1/2 + o(1)}$, $r(G) = n^{\alpha(G) + o(1)}$ and $p^2 n = C \log n$, it follows that
\begin{multline*}
\frac{\log \Ex\big[ Z(k,\ell,j,m,r) \big]}{\log n} \, \le \, \frac{k}{2} - m - \big( 1 -  \alpha(G) \big) r - C \bigg( \frac{3(k-m) - r}{2} - O\big( \sqrt{\delta}k \big) \bigg) + o(k)\\
 \, \le \, \bigg( \frac{1 - 3C}{2} \bigg) k \,-\, \bigg( \frac{2 - 3C}{2} \bigg) m \,+\, \bigg(  \alpha(G) - \frac{2 - C}{2} \bigg) r \,+\, O\big( C\sqrt{\delta}k \big) \, \le \, - \frac{\eps k}{4}.
\end{multline*}
Indeed, the second term is decreasing in $m$ for all $C \le 2/3$,\footnote{If $C \ge 2/3$ then simply note that the previous line is decreasing in $C$, since $3(k - m) - r \ge 2k - m \ge k$.} and we have (considering the cases $r = 0$ and $r = k$ separately) $\frac{1 - 3C}{2} \le - \eps/2$ and $\frac{1 - 3C}{2} + \alpha(G) - \frac{2 - C}{2} \le - \eps/3$, since (by assumption) we have $C \ge (1 + \eps) \max\big\{ 1/3, \alpha(G) - 1/2 \big\}$.\medskip

$(b)$ Next, suppose that $k \ge \delta n / r(G)$ but $k \le \sqrt{\delta} / p$. We partition the space according to the size of $e(\G_{\hat{S}})$: to be precise, we define $i = i\big(\hat{S}\big)$ by the inequalities
$$e(\G_{\hat{S}}) \in \left( \ds\frac{3(k - m) - r}{2} - \delta \big( 2 i \pm 1 \big) \big( k - m \big) \right) n.$$
Since $(1 - \delta)(k-m)n/2 \le e(\G_{\hat{S}}) \le (3(k-m)n - r)/2$ by Observation~\ref{obs:edge_lower_bound} and Proposition~\ref{prop:edge:counts}, we have $0 \le 2 \delta i(k-m) \le (1 + \delta)(k - m) - r/2$ for every set $\hat{S}$. Summing over $i$\footnote{The case $i = O(1)$ was already covered by the proof in part~$(a)$.}, applying Janson's inequality as in case~$(a)$, and using~\eqref{eq:eGhat}, we obtain
$$\Ex\big[ Z(k,\ell,j,m,r) \big] \, \le \, n^{O(\sqrt{\delta}k)} \sum_{i \ge 3} p^k \left( \frac{n}{k} \right)^{k-m- a_i/2} \exp\bigg( - \left( \ds\frac{3(k - m) - r}{2} - a_i \right) p^2 n \bigg),$$
where $a_i = 2\delta i(k-m)$. Substituting $p = n^{-1/2 + o(1)}$ and $p^2 n = C \log n$, and using the bound $k \ge n^{1 - \alpha(G) + o(1)}$, it follows that
$$\frac{\log \Ex\big[ Z(k,\ell,j,m,r) \big]}{\log n} \le  \max_a \bigg\{ - \frac{k}{2} \,+\, \alpha(G)\bigg( k - m - \frac{a}{2} \bigg) \,-\, C\bigg(\frac{3(k-m)-r}{2} - a \bigg) \bigg\} \,+\, O\big(\sqrt{\delta}k\big).$$

To bound the right-hand side, it suffices to check the extremal points. When $a = 0$, we note that $r \le k-m$ and $\alpha(G)-C \le 1/2-\eps/3$ to obtain a bound of
$$-k + 2(\alpha(G)-C)(k-m)+O\big(\sqrt{\delta}k\big) \le -\frac{\eps k}{4}.$$
At the other extreme, when $a = (1 + \delta)(k - m) - r/2$, we obtain analogously that
$$-k + (\alpha(G)-C)(k-m)+\frac{\alpha(G)k}{2}+O\big(\sqrt{\delta}k\big) \le -\frac{\eps k}{4}.$$

$(c)$ Finally, suppose that $k \ge \sqrt{\delta} / p$. Note first that $e(\G_{S,T,U}) \ge e(\G_{\hat{S}}) - O(k^2) = \Omega(kn)$. The inequality here is as in \eqref{eq:eGSTU}, whereas the equality is by Observation~\ref{obs:edge_lower_bound}. We thus have $$\frac{\mu}{\Delta} = O\left(\frac{n}{p\cdot e(\G_{S,T,U})}\right) = O\left(\frac{1}{\sqrt{\delta}}\right)\qquad \textrm{ and }\qquad \frac{\mu^2}{\Delta} = \Omega\left(p \cdot \frac{e(\G_{S,T,U})^2}{k^2n}\right) = \Omega(pn).$$ This follows because $\Delta = O(k^2 p^3 n)$, since $d(v) \le 3k$ for every $v \in V(\G_{S,T,U})$, and $\Delta = \Omega\big( p^3 e(\G_{S,T,U})^2 / n \big)$, by convexity. Janson's inequality then implies that
$$\Pr\Big( e\big( \G_{S,T,U}[A] \big) = 0 \Big) \, = \, e^{-\Omega\left(p n \sqrt{\delta}\right)},$$
from which it follows immediately that
\begin{align*}
\Ex\big[ Z(k,\ell,j,m,r) \big] \, & \le \, p^{k+\ell+j}\binom{n}{k}\binom{n}{\ell}\binom{n}{j}e^{-\Omega\left(p n \sqrt{\delta}\right)}
\\& \le \, p^{3k} \binom{n}{k}^3 e^{-\Omega\left(p n \sqrt{\delta}\right)} \, \le \, e^{-\Omega\left(p n \sqrt{\delta}\right)} \le \, e^{- 2\sqrt{n}},
\end{align*}
since $k \le \delta p n$. This completes the proof of~\eqref{eq:Zbounds:claim} in the case $r(G) \le \delta n$.

\medskip
\noindent \textbf{Case 2:} $r(G) \ge \delta n$.
\medskip

We now repeat the calculation above, replacing the bounds of Proposition~\ref{prop:edge:counts} with those of Proposition~\ref{prop:edge:counts:bigR}. Suppose first that $k \le \sqrt{\delta} / p$, and partition the space according to the maximum $s \in \{0,\ldots,k\}$ such that
$$e\big( \G_{\hat{S}} \big) \, \ge \, s \bigg( n - \frac{r(\O)}{2} \bigg).$$
By Proposition~\ref{prop:edge:counts:bigR}, there are at most $(12/\delta)^k (n/k)^s = O\big(n^{s + \sqrt{\delta}k}\big)$ such sets $S$ with $|S| = k$. Applying Janson's inequality, we obtain\footnote{When $s = k$, we trivially bound the number of sets $Z$ such that $e(\G_{\hat{S}}) \ge k \left(n - \frac{r(\O)}{2}\right)$ by $n^k$.}
$$\Ex\big[ Z(k,\ell,j,m,r) \big] \, \le \, n^{O(\sqrt{\delta}k)} \sum_{s = 0}^k p^k \cdot n^s \cdot \exp\bigg( - p^2 s \bigg( n - \frac{r(\O)}{2} \bigg) \bigg),$$
and hence
$$\frac{\log \Ex\big[ Z(k,\ell,j,m,r) \big]}{\log n} \, \le \, \max_s \bigg\{ s - C s \bigg( \frac{4 - \beta(G)}{4} \bigg) \bigg\} - \frac{k}{2} + O\big(\sqrt{\delta}k\big) \, \le \, - \frac{\eps k}{4}$$
since $C \ge (1+\eps)\cdot2 / \big( 4 - \beta(G) \big)$. The case $k \ge \sqrt{\delta} / p$ is exactly the same as case~$(c)$, above.

  Having bounded $\Ex\big[ Z(k,\ell,j,m,r) \big]$ in all cases, the result now follows easily by summing over $\ell$, $j$, $m$ and $r$. Indeed, by~\eqref{PatmostZ}, we have
$$\Pr\big(\B^\O_k(A) \big) \, \le \, \sum_{\ell = 0}^k \sum_{j=0}^{\ell} \sum_{m = 0}^{k/2} \sum_{r = 0}^{k - 2m}  \Ex\big[ Z(k,\ell,j,m,r) \big] \, \le \, \max\left\{ n^{-\delta k}, e^{-\sqrt{n}}\right\}$$
as claimed. This completes the proof of the lemma.
\end{proof}

In order to deduce the 1-statement in Theorem~\ref{thm:even:abelian} from Lemma~\ref{lem:BMSmethod}, we cannot simply apply the union bound over odd cosets $\O \in \SF(G)$, since an even-order abelian group $G$ can have as many as $|G|$ distinct maximum-size sum-free subsets. On the other hand, Lemma~\ref{lem:BMSmethod} (together with Theorem~\ref{thm:approx:stability}) does imply that the maximum-size sum-free subset of $A$ contains (with high probability) only $O(1)$ even elements, and moreover that any \emph{given} collection of $n^{o(1)}$ odd cosets are all likely to be `locally' maximal.

Motivated by these observations, it is natural to attempt to partition the odd cosets into two classes, depending on whether or not $|A \cap \O|$ is within $O(1)$ of $\max_{\O'} |A \cap \O'|$. However, the random variables $\{ |A \cap \O'| : \O' \in \SF(G) \}$ are highly correlated with one another, due to the large (size $n/2$) overlap between different odd cosets, and for this reason the maximum is not easy to control.\footnote{The behaviour of the random variable $\max_{\O'} |A \cap \O'|$ is in fact somewhat mysterious, and we believe that it merits further investigation.}

We resolve this problem by coupling with the hypergeometric distribution, for which the positive correlation between the variables $|A \cap \O|$ is greatly diminished. (In fact, these variables are roughly pairwise independent of one another.) For each $0 \le m \le 2n$, let $\Pr_m$ denote the probability measure on subsets of $G$ obtained by choosing each subset of size $m$ with equal probability. Note that, since any pair of distinct subgroups $\E,\E' \subset G$ of index $2$ intersect in a subgroup of index $4$, the information that $|A \cap \O| \ge a$ (and therefore $|A \cap \E| \le m - a$) has very little influence on the probability that $|A \cap \O'| \ge a$.

This crucial property of the hypergeometric distribution is captured by the following lemma. Given $k \in \N$ and an odd coset $\O \in \SF(G)$, define $M^\O_k(A)$ to be the event that $|A \cap \O| \ge k$, and let
$$X_k(A) \, := \, \sum_{\O \in \SF(G)} \mathbbm{1}\big[ M^\O_k(A) \big]$$
denote the number of odd cosets $\O \in \SF(G)$ for which $|A \cap \O| \ge k$.

\begin{lemma}\label{le:existsb}
Fix $\gamma > 0$ and $h\in\N$, and let $1 \ll m \le 2n$. There exists $b=b(G,m) \in [m]$ such that the following holds. If $A$ is chosen according to $\Pr_m$, then
\begin{enumerate}
\item[$(a)$] $\Ex\big[ X_b(A) \big] \le n^\gamma$ and \smallskip
\item[$(b)$] $X_{b+h}(A) \ge 1$ with high probability.
\end{enumerate}
\end{lemma}

The proof of Lemma~\ref{le:existsb} involves some straightforward but technical approximations of binomial coefficients, and so we defer it to an Appendix.

Let us denote by $\C_k^\O(A)$ the event that $|A \cap \O'| < |A \cap \O| + k$ for every $\O' \in \SF(G)$. We are now ready to complete the proof of our main theorem.

\begin{proof}[Proof of the $1$-statement in Theorem~\ref{thm:even:abelian}]
Let $\eps > 0$ be arbitrary, and let $0 < \delta < \delta_0(\eps)$ be sufficiently small and $n \ge n_0(\eps,\delta)$ be sufficiently large. Let $G$ be an abelian group with $2n$ elements, let $C \ge (1 + \eps)\lambda^{(\delta)}(G)$, set
$$p \, = \, \sqrt{ \frac{C \log n}{n} },$$
and let $A$ be a $p$-random subset of $G$. We shall prove that, with high probability as $n \to \infty$, we have $A \cap \O \in \SF(A)$ for some $\O \in \SF(G)$.

Indeed, let $B \in \SF(A)$ be a maximum-size sum-free subset of $A$, and note that, by Chernoff's inequality, and since $A \cap \O$ is sum-free for every $\O \in \SF(G)$, we have
\begin{equation}\label{eq:Chernoff:B}
|B| \, \ge \, \left( \frac{1}{2} - \delta \right) p |G|
\end{equation}
with high probability as $n \to \infty$. Therefore, applying Theorem~\ref{thm:approx:stability},  we deduce\footnote{Note that $p \ge C / \sqrt{n}$ since $n \ge n_0(\eps,\delta)$ is sufficiently large.} that, with high probability, we have $|B \setminus \O| \le \delta pn$ for some $\O \in \SF(G)$. Therefore,
\begin{align}
& \Pr_p\bigg( \bigcap_{\O \in \SF(G)} \big\{ A \cap \O \not\in \SF(A) \big\} \bigg) \, \le \, \Pr_p\bigg( \bigcup_{\O \in \SF(G)} \bigcup_{k=1}^{\delta p n} \Big( \B_k^\O(A) \cap \C_k^\O(A) \Big) \bigg) \,+\, o(1) \nonumber \\
& \hspace{1.5cm} \le \sum_{m = (1 - \delta^2) 2 p n}^{(1 + \delta^2) 2 p n} \Pr_m\bigg( \bigcup_{\O \in \SF(G)} \bigcup_{k=1}^{\delta p n} \Big( \B_k^\O(A) \cap \C_k^\O(A) \Big) \bigg) \cdot \Pr_p\big( |A| = m \big) \,+\, o(1),\label{eq:coupling:hyper}
\end{align}
where we again used Chernoff's inequality. Let $b = b(G,m) \in [m]$ be given by Lemma~\ref{le:existsb} (with $h = 1/\delta^2$) so, with high probability, we have $|A \cap \O'| \ge b + 1/\delta^2$ for some $\O' \in \SF(G)$. Note that if such an $\O'$ exists, then $\C_k^\O(A)$ implies that either $|A \cap \O| \ge b$ or $k \ge 1/\delta^2$.

Let us first bound the probability when $k \ge 1/\delta^2$. Indeed, by Hoeffding's inequality (see, e.g.,~\cite{Chvatal}), we have
\begin{equation}\label{eq:probB:klarge1}
\Pr_m\big( \B^\O_k(A) \big)  \, = \, \sum_{i = m/2 - \delta^2 m}^{m/2 + \delta^2 m} \Pr_m\Big( \B^\O_k(A) \,\big|\, |A \cap \E| = i \Big) \Pr_m\big( |A \cap \E| = i \big) + o\bigg( \frac{1}{n^3} \bigg).
\end{equation}
Moreover the event $\B^\O_k(A)$ is increasing in $A \cap \E$ and decreasing in $A \cap \O$, and therefore (recalling from Lemma~\ref{lem:BMSmethod} the definition of $\Pr_{p^\pm}$), we have
\begin{align}
\Pr_m\Big( \B^\O_k(A) \,\big|\, |A \cap \E| = i \Big) & \, \le \, \Pr_{p^\pm}\Big( \B^\O_k(A) \,\Big|\, \big( |A \cap \E| \ge i \big) \cap \big( |A \cap \O| \le m - i \big) \Big) \nonumber\\
& \, \le \, 2 \cdot \Pr_{p^\pm}\big( \B^\O_k(A) \big) \, \le \, 2 \cdot n^{-1/\delta} \, \ll \, \frac{1}{n^3} \label{eq:probB:klarge2}
\end{align}
for every $k \ge 1/\delta^2$, by Lemma~\ref{lem:BMSmethod}. Indeed, the first inequality follows since $p^\pm$ chooses sets $A$ uniformly given $|A \cap \E|$ and $|A \cap \O|$. To see the second inequality, simply note that $\Pr_{p^\pm}\big( ( |A \cap \E| \ge i ) \cap ( |A \cap \O| \le m - i ) \big) \ge 1/2$ for every $i \le m/2 + \delta^2 m \le pn + 3\delta^2 pn$.

Next, let us bound the probability when $|A \cap \O| \ge b$. Similarly to above, we have
$$\Pr_m\Big( \B^\O_k(A) \cap \big( |A \cap \O| \ge b \big) \Big) \,=\, \sum_{i = 0}^{m - b} \Pr_m\Big( \B^\O_k(A) \,\big|\, |A \cap \E| = i \Big) \cdot \Pr_m\big( |A \cap \E| = i \big),$$
and moreover
$$\Pr_m\Big( \B^\O_k(A) \,\big|\, |A \cap \E| = i \Big) \, \le \, 2 \cdot \Pr_{p^\pm}\big( \B^\O_k(A) \big) \, \le \, 2 \cdot n^{-\delta},$$
for every $k \ge 1$, by~\eqref{eq:probB:klarge2} and Lemma~\ref{lem:BMSmethod}, and
$$\Ex_m\big[ X_b(A) \big] \, = \, \sum_{\O \in \SF(G)} \sum_{i = 0}^{m - b} \Pr_m\big( |A \cap \E| = i \big) \, \le \, n^{\delta/2},$$
by Lemma~\ref{le:existsb}$(a)$. Therefore
\begin{equation}\label{eq:probB:ksmall}
\sum_{\O \in \SF(G)} \Pr_m\Big( \B^\O_k(A) \cap \big( |A \cap \O| \ge b \big) \Big) \, \le \, 2 \cdot n^{-\delta/2}
\end{equation}
for evey $k \ge 1$. Combining~\eqref{eq:probB:klarge1},~\eqref{eq:probB:klarge2} and~\eqref{eq:probB:ksmall}, it follows that
$$\Pr_m\bigg( \bigcup_{\O \in \SF(G)} \bigcup_{k=1}^{\delta p n} \Big( \B_k^\O(A) \cap \C_k^\O(A) \Big) \bigg) \le \, 2 \cdot \sum_{k = 1}^{1/\delta^2}  n^{-\delta/2} + \sum_{k = 1/\delta^2}^{\delta p n} \sum_{\O \in \SF(G)} \frac{1}{n^3} \, \le \, n^{-\delta/3}$$
for every $m \in (1 \pm \delta^2) 2 p n$, and every sufficiently large $n$. Hence, by~\eqref{eq:coupling:hyper}, we have
$$\Pr_p\bigg( \bigcap_{\O \in \SF(G)} \big\{ A \cap \O \not\in \SF(A) \big\} \bigg) \, = \, o(1),$$
as required.
\end{proof}

\appendix

\section{Lemmas on the hypergeometric distribution}

In this appendix we will prove Lemmas~\ref{lem:M:nice} and~\ref{le:existsb}. We begin with the latter.

\subsection{Proof of Lemma~\ref{le:existsb}}

We are required to prove that there exists $b = b(G,m) \in [m]$ with the following properties: at most $n^{o(1)}$ odd cosets are expected to contain at least $b$ elements of $A$, but with high probability some odd coset contains at least $b + \omega$ elements of $A$, where $\omega \to \infty$ as $n \to \infty$. For the proof, it will be convenient to shift the notation by $m/2$ as follows: For each $k \in \N$ and each $\O \in \SF(G)$, let us denote by $M^\O_k(A)$ the event that $|A \cap \O| \ge m/2 + k$, and by
$$X_k(A) \, = \, \sum_{\O \in \SF(G)} \mathbbm{1}\big[ M^\O_k(A) \big]$$
the number of odd cosets $\O \in \SF(G)$ for which $|A \cap \O| \ge m/2 + k$.

The main step in the proof of Lemma~\ref{le:existsb} is the following bound on the correlation between the events $M^\O_k(A)$. Here, and throughout this Appendix, we write $x\sim y$ to mean that $x/y\to 1$ under the given asymptotics.

\begin{lemma}\label{2ndmom}
Let $\O,\O' \in \SF(G)$ be distinct odd cosets, and let $k,m \in \N$ be such that $1 \ll k \ll m \ll k^2$. Then
\[
\Pr_m\big( M^\O_k(A) \cap M^{\O'}_k(A) \big) \sim \Pr_m\big( M^\O_k(A) \big)^2
\]
as $n \to \infty$.
\end{lemma}

We begin by calculating $\Pr_m\big( M^\O_k(A) \big)$ asymptotically, using the following simple bounds.

\begin{lemma}\label{le:hypgeoest}
Let $a,b,N \in \N$ with $b^{3/2} \ll a \ll N$. Then
\[
\frac{\binom{N}{a+b}\binom{N}{a-b}}{\binom{2N}{2a}} \sim \frac{1}{\sqrt{\pi a}}\exp\left(-\frac{b^2}{a}\right)
\]
as $a,N \to \infty$.
\end{lemma}

\begin{proof}
This is nothing more than an application of Stirling's formula
\[
n! \sim \sqrt{2\pi n}\left(\frac{n}{e}\right)^n,
\]
and the partial Taylor series
\[
\left|\log(1+x)-x+\frac{x^2}{2}\right| \le O\big(|x|^3\big),
\]
which is valid for all sufficiently small $|x|$.
\end{proof}

Let us denote by $\hat{M}^\O_x(A)$ the event that $|A \cap \O| = m/2 + x$, so $M^\O_k(A) = \ds\bigcup_{x \ge k} \hat{M}^\O_x(A)$.

\begin{lemma}\label{le:expX}
For every $\O \in \SF(G)$,
\[
\Pr_m\big( M^\O_k(A) \big) \sim \sqrt{\frac{2}{\pi m}} \sum_{x \ge k} \exp\left(-\frac{2x^2}{m}\right).
\]
\end{lemma}

\begin{proof}
Observe that
\[
\Pr_m\big( M^\O_k(A) \big) \, = \, \sum_{x \ge k} \Pr_m\big( \hat{M}^\O_x(A) \big) \, = \, \sum_{ x \ge k} \frac{\binom{n}{m/2+x}\binom{n}{m/2-x}}{\binom{2n}{m}}.
\]
The result now follows by applying Lemma~\ref{le:hypgeoest} with $N=n$, $a=m/2$ and $b=x$.
\end{proof}

The following bounds now follow easily.

\begin{lemma}\label{le:expXtheta}
For every $\O \in \SF(G)$,
\[
\Pr_m\big( M^\O_{k}(A) \big) \, = \, \Theta\left( \frac{\sqrt{m}}{k} \exp\left(-\frac{2k^2}{m}\right) \right).
\]
\end{lemma}

\begin{proof}
By Lemma~\ref{le:expX}, we have
\[
\Pr_m\big( M^\O_{k}(A) \big) \, = \, \Theta\left( \frac{1}{\sqrt{m}} \exp\left(-\frac{2k^2}{m}\right) \sum_{x\ge 0} \exp\left(-\frac{4kx}{m}-\frac{2x^2}{m}\right) \right).
\]
Now, the asymptotics $k \ll m \ll k^2$ imply that
\[
\sum_{x \ge 0} \exp\left(- \frac{4kx}{m} - \frac{2x^2}{m}\right) \,=\, \Theta\bigg( \frac{m}{k} \bigg),
\]
and the lemma follows immediately.
\end{proof}

When bounding the probability of $M^\O_k(A) \cap M^{\O'}_k(A)$, the following notation will be useful. Set
\[
\Lambda := \big\{ (x,y,z) \in \Z^3 \,:\, x+y \ge k, \,  x+z \ge k \big\},
\]
and given $\O, \O' \in \SF(G)$ and $x,y,z \in \Z$, denote by $\hat{M}^{\O,\O'}_{x,y,z}(A)$ the event that
$$|A \cap \O \cap \O'| = \frac{m}{4} + x, \quad |A \cap \O \cap \E' | = \frac{m}{4} + y, \quad \textup{and}  \quad |A \cap \O' \cap \E | = \frac{m}{4} + z,$$
where as usual $\E = G \setminus \O$ and $\E' = G \setminus \O'$.

\begin{lemma}\label{le:varX}
Let $\O,\O' \in \SF(G)$ be distinct odd cosets. Then
\[
\Pr_m\big( M^\O_k(A) \cap M^{\O'}_k(A) \big) \, \sim \, \frac{4\sqrt{2}}{(\pi m)^{3/2}} \sum_{(x,y,z)\in\Lambda} \exp\left(-\frac{2}{m}\Big( (x+y)^2+(x+z)^2+(y+z)^2 \Big) \right).
\]
\end{lemma}

\begin{proof}
Note first that
\[
\Pr_m\big( \hat{M}^{\O,\O'}_{x,y,z}(A) \big) \, = \, \frac{\binom{n}{m/2+x+y}\binom{n}{m/2-x-y}}{\binom{2n}{m}} \frac{\binom{n/2}{m/4+x}\binom{n/2}{m/4+y}}{\binom{n}{m/2+x+y}} \frac{\binom{n/2}{m/4+z}\binom{n/2}{m/4-x-y-z}}{\binom{n}{m/2-x-y}}.
\]
By Lemma~\ref{le:hypgeoest}, this is asymptotically equal to
\[
\sqrt{\frac{2}{\pi m}}\exp\left(-\frac{2(x+y)^2}{m}\right) \sqrt{\frac{4}{\pi m}}\exp\left(-\frac{(x-y)^2}{m}\right) \sqrt{\frac{4}{\pi m}}\exp\left(-\frac{(x+y+2z)^2}{m}\right),
\]
and this expression is equal to
\[
\frac{4\sqrt{2}}{(\pi m)^{3/2}} \exp\left(-\frac{2}{m} \Big( (x+y)^2+(x+z)^2+(y+z)^2 \Big) \right).
\]
Thus,
\begin{align*}
& \Pr_m\big( M^\O_k(A) \cap M^{\O'}_k(A) \big) \,=\, \sum_{(x,y,z)\in\Lambda} \Pr_m\big( \hat{M}^{\O,\O'}_{x,y,z}(A) \big) \\
& \hspace{2.5cm} \,\sim\, \frac{4\sqrt{2}}{(\pi m)^{3/2}} \sum_{(x,y,z)\in\Lambda} \exp\left(-\frac{2}{m}\Big( (x+y)^2+(x+z)^2+(y+z)^2 \Big) \right),
\end{align*}
as claimed.
\end{proof}

We are almost ready to prove Lemma~\ref{2ndmom}; we need one more well-known fact.

\begin{fact}\label{fa:int}
$$\sum_{x\in\Z} \exp\left(-\frac{2x^2}{m}\right) \sim \sqrt{\frac{\pi m}{2}}$$
as $m \to \infty$.
\end{fact}

\begin{proof}[Proof of Lemma~\ref{2ndmom}]
Observe that $(a,b,c)$ is equal to $(x+y,x+z,y+z)$ for some triple $(x,y,z)$ if and only if $a+b+c$ is even and
\[
(x,y,z)=\left(\frac{a+b-c}{2},\frac{c+a-b}{2},\frac{b+c-a}{2}\right).
\]
Letting
\[
\Lambda' := \big\{ (a,b,c) \in \Z^3 \,:\, a\ge k, \, b\ge k, \, a+b+c \text{ even} \big\},
\]
it follows that
$$\sum_{(x,y,z)\in\Lambda} \exp\left(-\frac{2}{m}\Big( (x+y)^2+(x+z)^2+(y+z)^2 \Big)\right) = \sum_{(a,b,c)\in\Lambda'} \exp\left(-\frac{2}{m} \big( a^2+b^2+c^2 \big) \right).$$
We may split up the right-hand side into separate sums according to the parity of $a+b$, and hence of $c$. Doing this, we may rewrite the sum as
\begin{align*}
& \sum_{\substack{a\ge k,\, b\ge k, \\ a+b\text{ even}}} \exp\left(-\frac{2(a^2+b^2)}{m}\right) \sum_{c\text{ even}} \exp\left(-\frac{2c^2}{m}\right) \\
& \hspace{4cm} + \sum_{\substack{a\ge k,\, b\ge k, \\ a+b\text{ odd}}} \exp\left(-\frac{2(a^2+b^2)}{m}\right) \sum_{c\text{ odd}} \exp\left(-\frac{2c^2}{m}\right).
\end{align*}
Since $m$ is large, we have
\[
\sum_{c\text{ odd}} \exp\left(-\frac{2c^2}{m}\right) \sim \sum_{c\text{ even}} \exp\left(-\frac{2c^2}{m}\right) \sim \frac{1}{2}\sum_c \exp\left(-\frac{2c^2}{m}\right) \sim \frac{1}{2}\sqrt{\frac{\pi m}{2}},
\]
where we have used Fact \ref{fa:int} for the final estimate. We also have
\[
\sum_{a\ge k,\, b\ge k} \exp\left(-\frac{2(a^2+b^2)}{m}\right) = \Bigg( \sum_{a\ge k} \exp\left(-\frac{2a^2}{m}\right) \Bigg)^2 \sim \frac{\pi m}{2} \cdot \Pr_m\big( M^\O_k(A) \big)^2
\]
for an arbitrary odd coset $\O \in \SF(G)$, by Lemma \ref{le:expX}. Putting all this together, we conclude that
\begin{equation}
\label{eq:exponential:sum}
\sum_{(x,y,z)\in\Lambda} \exp\left(-\frac{2}{m} \Big( (x+y)^2+(x+z)^2+(y+z)^2 \Big) \right) \sim \frac{(\pi m)^{3/2}}{4\sqrt{2}} \cdot \Pr_m\big( M^\O_k(A) \big)^2.
\end{equation}
We may now use our estimate for $\Pr_m\big( M^\O_k(A) \cap M^{\O'}_k(A) \big)$ from Lemma~\ref{le:varX}. Together with~\eqref{eq:exponential:sum}, this implies that
\[
\Pr_m\big( M^\O_k(A) \cap M^{\O'}_k(A) \big) \sim \Pr_m\big( M^\O_k(A) \big)^2,
\]
as required.
\end{proof}

Lemma~\ref{le:existsb} now follows by a straightforward application of the second moment method. For completeness we give the details.

\begin{lemma}\label{le:existsO}
If $\Exp\big[ X_k \big] \gg 1$, then $X_k \ge 1$ with high probability.
\end{lemma}

\begin{proof}
We have
\begin{align*}
\Var\big( X_k \big) & \, = \, \Exp\big[ X_k^2 \big] - \Exp\big[ X_k \big]^2 \, = \, \sum_{\O,\O' \in \SF(G)} \Pr_m\big( M^\O_k(A) \cap M^{\O'}_k(A) \big) - \Exp\big[ X_k \big]^2 \\
&\, = \,  \Exp\big[ X_k \big] + \sum_{\O\neq\O'} \Pr_m\big( M^\O_k(A) \cap M^{\O'}_k(A) \big) - \Exp\big[ X_k \big]^2 \\
&\, = \,  \Exp\big[ X_k \big] + \big( 1+o(1) \big) \sum_{\O \neq \O'} \Pr_m\big( M^\O_k(A) \big)^2 - \Exp\big[ X_k \big]^2,
\end{align*}
by Lemma~\ref{2ndmom}. Therefore,
\[
\Var\big( X_k \big) \,\le\, \Exp\big[ X_k \big] + \big( 1+o(1) \big)\Exp\big[ X_k \big]^2 - \Exp\big[ X_k \big]^2 \,=\, o\big( \Exp\big[ X_k \big]^2 \big).
\]
Hence, by Chebyshev's inequality, we have $X_k \ge 1$ with high probability as $n \to \infty$.
\end{proof}

It only remains to show that $\Exp\big[ X_k \big]$ does not decay too quickly.

\begin{lemma}\label{le:expXdiff}
For every constant $h > 0$, we have
\[
\big| \Exp\big[ X_k \big] - \Exp\big[ X_{k+h} \big] \big| \, = \, o\big( \Exp\big[ X_k \big] \big).
\]
\end{lemma}

\begin{proof}
By Lemma~\ref{le:expXtheta}, we have
$$\Exp X_k \,=\, \Omega\left( \frac{r(G)\sqrt{m}}{k} \exp\left(-\frac{2k^2}{m}\right) \right).$$
whereas, by Lemma~\ref{le:expX}, we have
$$\Exp\big[ X_k \big] - \Exp\big[ X_{k+h} \big] \, = \, O\left( \frac{r(G)}{\sqrt{m}} \sum_{x=k}^{k+h} \exp\left(-\frac{2x^2}{m}\right) \right) \, = \, O\left( \frac{r(G)}{\sqrt{m}} \exp\left(-\frac{2k^2}{m}\right) \right).$$
Since we assumed that $k\ll m$, the lemma follows.
\end{proof}

\begin{proof}[Proof of Lemma \ref{le:existsb}]
If $r(G) \le n^\gamma$ then the lemma is trivial (set $b = 0$), so assume that $r(G) > n^\gamma$ and let $b = b(G,m)$ be minimal such that $\Ex\big[ X_b(A) \big] \le n^\gamma$. It follows that $\Ex\big[ X_{b+h}(A) \big] \gg 1$, by Lemma~\ref{le:expXdiff}, and hence that $X_{b + h}(A) \ge 1$ with high probability, by Lemma~\ref{le:existsO}, as required.
\end{proof}

\subsection{Proof of Lemma~\ref{lem:M:nice}}

Let $G$ be an even-order abelian group, and note that the lemma is trivial if $r(G) \le \delta n$. Recall that $\M$ denotes the collection of odd cosets $\O \in \SF(G)$ such that $|A \cap \O|$ is maximal. We are required to prove that with high probability there is an $\O \in \M$ such that $\E = G \setminus \O$ is nice. This is an immediate consequence of the following lemma. Recall that $\omega = \omega(n)$ is a function such that $\omega \to \infty$ slowly as $n \to \infty$.

\begin{lemma}
With high probability, the following hold:
\begin{itemize}
\item[$(a)$] $|A \cap \O| \le pn + \omega \sqrt{pn}$ for every subgroup $\E = G \setminus \O$ which is not nice.
\item[$(b)$] There exists a nice subgroup $\E = G \setminus \O$ such that $|A \cap \O| \ge pn + \omega\sqrt{pn}$.
\end{itemize}
\end{lemma}

\begin{proof}
Part~$(a)$ follows from Chernoff's inequality and the union bound, since there are at most $O(1/\delta)$ subgroups that are not nice. To prove part~$(b)$, we again couple with the hypergeometric distribution, and apply Lemma~\ref{le:existsb}. Indeed, we have $|A| \ge 2pn - \omega\sqrt{pn}$ with high probability, and for each $m \ge 2pn - \omega\sqrt{pn}$ there exists a $b = b(G,m)$ such that  $\Ex\big[ X_b(A) \big] \le \sqrt{n}$ and $X_b(A) \ge 1$ with high probability in $\Pr_m$. But, by Lemma~\ref{le:expXtheta}, we have $\Ex\big[ X_b(A) \big] = n^{1 + o(1)}$ for $b = pn + \omega\sqrt{pn}$, and so we are done.
\end{proof}

\section*{Acknowledgements}

This research was begun during a visit of N.B. to IMPA in January and February 2013. He is grateful for their hospitality and support, and for many fruitful discussions. The authors would also like to thank Nathan Kettle for helpful conversations, and the referee for a very careful reading of the proof, and for numerous suggestions which improved the presentation.


\begin{thebibliography}{99}

\bibitem{ABMS1} N.~Alon, J.~Balogh, R.~Morris and W.~Samotij, Counting sum-free subsets in Abelian groups, to appear in \emph{Israel J. Math.}

\bibitem{ABMS2} N.~Alon, J.~Balogh, R.~Morris and W.~Samotij, A refinement of the Cameron-Erd\H{o}s conjecture, \emph{Proc. London Math. Soc.}, \textbf{108} (2014), 44--72.

\bibitem{AK} N.~Alon and D.J.~Kleitman,  Sum-free subsets, in A tribute to Paul Erd\H{o}s (A.~Baker, B.~Bollob\'as and A.~Hajnal, eds), Cambridge University Press, Cambridge, 1990, 13--26.

\bibitem{AS} N.~Alon and J.~Spencer, The Probabilistic Method (3rd edition), Wiley Interscience, 2008.

\bibitem{BSS} L.~Babai, M.~Simonovits and J.~Spencer, Extremal subgraphs of random graphs, \emph{J. Graph Theory}, \textbf{14} (1990), 599--622.

\bibitem{BMS1} J.~Balogh, R.~Morris and W.~Samotij, Random sum-free subsets of abelian groups, to appear in \emph{Israel J. Math.}

\bibitem{BMS2} J.~Balogh, R.~Morris and W.~Samotij, Independent sets in hypergraphs, submitted.

\bibitem{BMSW} J.~Balogh, R.~Morris, W.~Samotij and L.~Warnke, The typical structure of sparse $K_{r+1}$-free graphs, submitted.

\bibitem{BT} B.~Bollob\'as and A.~Thomason, Threshold functions, \emph{Combinatorica}, \textbf{7} (1986), 35--38.

\bibitem{Chvatal} V.~Chv\'atal, The tail of the hypergeometric distribution, \emph{Discrete Math.},
\textbf{25} (1979), 285--287.

\bibitem{CG} D.~Conlon and W.T.~Gowers, Combinatorial theorems in sparse random sets, submitted.

\bibitem{DMK} B.~DeMarco and J.~Kahn, Mantel's theorem for random graphs, submitted.

\bibitem{DY} P.H.~Diananda and H.P.~Yap, Maximal sum-free sets of elements of finite groups, \emph{Proc. Japan Acad.}, \textbf{45} (1969), 1--5.

\bibitem{FR86} P.~Frankl and V.~R\"odl, Large triangle-free subgraphs in graphs without $K_4$, \emph{Graphs Combin.}, \textbf{2} (1986), 135--144.

\bibitem{Friedgut} E.~Friedgut, Sharp thresholds of graph properties, and the $k$-sat problem, with an appendix by Jean Bourgain, \emph{J. Amer. Math. Soc.}, \textbf{12} (1999), 1017--1054.

\bibitem{FRRT} E.~Friedgut, V.~R\"odl, A.~Ruci\'nski and P.~Tetali, A sharp threshold for random graphs with a monochromatic triangle in every edge coloring, \emph{Mem. Amer. Math. Soc.}, \textbf{179} (2006), 66pp.

\bibitem{FRS} E.~Friedgut, V.~R\"odl and M.~Schacht, Ramsey properties of random discrete structures, \emph{Random Structures Algorithms}, \textbf{37} (2010), 407--436.

\bibitem{GRR} R.~Graham, V.~R\"odl and A.~Ruci\'nski, On Schur properties of random subsets of integers, \emph{J. Number Theory}, \textbf{61} (1996), 388--408.

\bibitem{Green} B.~Green, The Cameron-Erd\H{o}s conjecture, \emph{Bull. London Math. Soc.}, \textbf{36} (2004), 769--778.

\bibitem{GR} B.~Green and I.Z.~Ruzsa, Sum-free sets in abelian groups, \emph{Israel J. Math}, \textbf{147} (2005), 157--189.

\bibitem{Hatami} H.~Hatami, A structure theorem for Boolean functions with small total influences, \emph{Ann. Math.}, \textbf{176} (2012), 509--533.

\bibitem{JLR} S.~Janson, T.~\L uczak and A.~Ruci\'nski, Random Graphs, Wiley, 2000.

\bibitem{KLR1} Y.~Kohayakawa, T.~\L uczak and V.~R\"odl, Arithmetic progressions of length three in subsets of a random set, \emph{Acta Arith.}, \textbf{75} (1996), 133--163.

\bibitem{LLS} V.F.~Lev, T.~\L uczak and T.~Schoen, Sum-free sets in abelian groups, \emph{Israel J. Math.}, \textbf{125} (2001), 347--367.

\bibitem{RR1} V.~R\"odl and A.~Ruci\'nski, Threshold functions for Ramsey properties, \emph{J. Amer. Math. Soc.}, \textbf{8} (1995), 917--942.

\bibitem{RR2} V.~R\"odl and A.~Ruci\'nski, Rado partition theorem for random subsets of integers, \emph{Proc. London Math. Soc.}, \textbf{74} (1997), 481--502.

\bibitem{Sam} W.~Samotij, Stability results for random discrete structures, to appear in \emph{Random Structures Algorithms}.stat271

\bibitem{Sap02} A.A.~Sapozhenko, Asymptotics of the number of sum-free sets in abelian groups of even order, (Russian) \emph{Dokl. Akad. Nauk.}, \textbf{383} (2002), 454--457.

\bibitem{Sap03} A.A.~Sapozhenko, The Cameron-Erd\H{o}s Conjecture (Russian), \emph{Dokl. Akad. Nauk.}, \textbf{393} (2003), 749--752.

\bibitem{ST} D.~Saxton and A.~Thomason, Hypergraph containers, submitted.

\bibitem{Sch} M.~Schacht, Extremal results for random discrete structures, submitted.

\bibitem{Schur} I.~Schur, Uber die Kongruenz $x^m + y^m \equiv z^m \pmod p$, \emph{Jahresber. Deutsche Math.-Verein.}, \textbf{25} (1916), 114--117.

\bibitem{Warnke} L.~Warnke, On the method of typical bounded differences, submitted.

\end{thebibliography}
\end{document}